\documentclass[a4paper,10pt]{amsart}

\usepackage[colorlinks=printing]{hyperref}
\usepackage{amsmath,amssymb,amsmath,esint,graphicx,subfigure}

\newtheorem{theorem}{Theorem}[section]
\newtheorem{lemma}[theorem]{Lemma}
\newtheorem{proposition}[theorem]{Proposition}
\newtheorem{corollary}[theorem]{Corollary}
\newtheorem{definition}{Definition}[section]

\theoremstyle{remark} \newtheorem{remark}[theorem]{Remark}
\theoremstyle{definition} 

\numberwithin{equation}{section}

\newcommand{\dx}{\Delta x}
\newcommand{\dt}{\Delta t}
\newcommand{\del}{\partial}

\newcommand{\ra}{\rightarrow}

\newcommand{\alp}{\alpha}
\newcommand{\eps}{\ensuremath{\varepsilon}}
\newcommand{\R}{\ensuremath{\mathbb{R}}}

\newcommand{\Z}{\ensuremath{\mathbb{Z}}}

\newcommand{\Div}{\mathrm{div}}
\newcommand{\Levy}{\ensuremath{\mathcal{L}}}

\newcommand{\sgn}{{\rm sgn}\, }

\newcommand{\dif}{\mathrm{d}}
\newcommand{\xl}{x_{|x_l=x_{\alp_l}}}

\begin{document}

\title[Numerical methods for convection-diffusion equations]{On
  numerical methods and error estimates\\ for degenerate  fractional convection-diffusion
  equations} 

\author[{S.~Cifani}]{{Simone Cifani}}
\address[Simone Cifani]{\\ Department of Mathematics\\ Norwegian University of Science and Technology (NTNU)\\
 N-7491 Trondheim, Norway.}

\author[{E.~R.~Jakobsen}]{{Espen R. Jakobsen}}
\address[Espen R. Jakobsen]{\\ Department of Mathematics\\
Norwegian University of Science and Technology (NTNU)\\
N-7491 Trondheim, Norway.} \email[]{erj\@@math.ntnu.no}
\urladdr{http://www.math.ntnu.no/\~{}erj/}

\keywords{Fractional conservation laws, 
convection-diffusion equations, porous medium equation, entropy
solutions, numerical method, convergence rate, error estimates}

\thanks{This research was supported by the Research Council of Norway (NFR) through the project "Integro-PDEs:
numerical methods, analysis, and applications to finance".}

\begin{abstract}
First we introduce and analyze a convergent numerical method for a large
class of nonlinear nonlocal possibly degenerate convection diffusion
equations. Secondly we develop a new Kuznetsov type theory and obtain
general and possibly optimal error estimates for our numerical methods
-- even when 
the principal derivatives have any fractional order between 1 and 2!
The class of equations we 
consider includes equations with nonlinear and possibly degenerate
fractional or general Levy diffusion. Special cases are conservation
laws, fractional conservation laws, certain fractional porous medium
equations, and new strongly degenerate equations. 
\end{abstract}

\maketitle

\section{Introduction}
In this paper we develop a numerical method along with a general
Kuznetsov type theory of error estimates for integro partial
differential equations of the form 
\begin{equation}\label{1}
\left\{
\begin{array}{ll}
\partial_tu+\Div f(u)=\Levy^\mu[A(u)],&(x,t)\in Q_T,\\
u(x,0)=u_{0}(x), &x\in\R^d,
\end{array}
\right.
\end{equation}
where $Q_T=\mathbb{R}^d\times(0,T)$ and the nonlocal diffusion operator
$\Levy^\mu$ is  defined as
\begin{align}\label{non-loc-op}
\Levy^\mu[\phi](x)=\int_{|z|>0}\phi(x+z)-\phi(x)-z\cdot\nabla\phi(x)\,\mathbf{1}_{|z|<1}(z)\
\dif\mu(z),
\end{align}
for smooth bounded functions $\phi$. Here $\mathbf{1}$
denotes the indicator function. 
Throughout the paper the data $(f,A,\mu,u_0)$ is assumed
to satisfy:
\begin{itemize}
\item[(\textbf{A}.1)] $f=(f_1,\ldots,f_d)\in W^{1,\infty}(\R;\R^d)$ with $f(0)=0$,
\medskip
\item[(\textbf{A}.2)] $A\in W^{1,\infty}(\R)$, $A$ non-decreasing with
$A(0)=0$,
\medskip
\item[(\textbf{A}.3)] $\mu\geq 0$ is a Radon measure such that $\int_{|z|>0}|z|^2\wedge 1 \
\dif\mu(z)<\infty$,
\medskip
\item[(\textbf{A}.4)] $u_0\in L^\infty(\R^d)\cap L^1(\R^d)\cap BV(\R^d)$.
\end{itemize}
We use the notation $a\wedge b=\min(a,b)$ and $a\vee b=\max(a,b)$. 

\begin{remark}
These assumptions can be relaxed in two standard ways: (i) $f,A$ can take any
value at $u=0$ (replace $f$ by $f-f(0)$ etc.), and (ii) $f,A$ can
 be assumed to be locally Lipschitz.
By the maximum principle and (\textbf{A}.4), solutions of
\eqref{1} are bounded, and locally Lipschitz functions are
Lipschitz on compact domains.
\end{remark}

The measure $\mu$ and the operator
$\Levy^\mu$ are respectively the L\'{e}vy measure and the generator of
a pure jump L\'{e}vy process. Any such process has a L\'evy
measure and generator satisfying \eqref{non-loc-op} and
(\textbf{A}.3), see e.g. \cite{App:Book}. Example are the
symmetric $\alp$-stable processes with fractional Laplace generators where
\begin{equation}\label{frac_lap}
\begin{split}
d\mu(z)=c_\lambda\frac{dz}{|z|^{d+\lambda}}\ \
(c_\lambda>0)\qquad\text{and}\qquad \Levy^\mu\equiv
-(-\Delta)^{\lambda/2}\quad\text{for }\lambda\in(0,2).
\end{split}
\end{equation}
Non-symmetric examples are popular in mathematical finance, e.g. the CGMY
model where
\begin{equation*}
d\mu(z)=\left\{
\begin{split}
&\frac{C\,e^{-G|z|}}{|z|^{1+\lambda}}dz&\text{for $z>0$,}\\
&\frac{C\,e^{-M|z|}}{|z|^{1+\lambda}}dz&\text{for $z<0$,}
\end{split}
\right.
\end{equation*}
and where $d=1$, $\lambda(=Y)\in(0,2)$, and $C,G,M>0$. We refer the reader to
\cite{Cont/Tankov} for more details on this and other
nonlocal models in finance. In both examples the nonlocal operator
behaves like a fractional derivative of order between 0 and 2.

Equation \eqref{1} has a local non-linear convection term (the
$f$-term) and a fractional (or nonlocal) non-linear possibly
degenerate diffusion 
term (the $A$-term). Special cases are scalar conservation laws
($A\equiv0$), fractional and L\'evy conservation laws ($A(u)=u$ and
$\alp$-stable or more general $\mu$) -- see
e.g. \cite{BiKaWo99,Alibaud} and \cite{BiKaWo00,RoYo07,KaUl10}, fractional
porous medium equations \cite{DPQRV} ($A=|u|^{m-1}u$ for $m\geq1$ and
$\alp$-stable $\mu$), and strongly degenerate equations where
$A$ vanishes on a set of positive measure. If either $A$ is degenerate
or $\Levy^\mu$ is a fractional derivative of order less than $1$, then
solutions of \eqref{1} are not smooth in general and uniqueness fails
for weak (distributional) solutions. Uniqueness can be regained by
imposing additional entropy conditions in a similar way to what is
done for conservation laws. The Kruzkov entropy solution theory of
scalar conservation laws \cite{Kru70} was extended to cover fractional
conservation laws in \cite{Alibaud}, to more general L\'evy conservation
laws in \cite{KaUl10}, and then finally to setting of this paper,
equations with non-linear fractional diffusion and general L\'evy
measures in \cite{Cifani/Jakobsen}. For local 2nd order degenerate
convection diffusion equations like 
\begin{equation}\label{locCD}
\partial_tu+\Div f(u)=\Delta A(u),
\end{equation}
there is an entropy solution theory due to Carrillo 
\cite{Car99}. 

In recent years, integro partial differential equations like \eqref{1}
have been at the center of a very active field of research. 
A thorough description of the mathematical background for
such equations, relevant bibliography, and applications to several
disciplines of interest can be found in
\cite{Alibaud,Alibaud/Cifani/Jakobsen,BiKaWo00,Cifani/Jakobsen,DPQRV,KaUl10}.

The first contribution of this paper is to introduce a
numerical method for equation \eqref{1} and prove that it converges toward
the entropy solution of \eqref{1} under assumptions
$(\mathbf{A}.1)$--$(\mathbf{A}.4)$. The numerical method is based upon
a monotone finite volume discretization of an approximate equation with
truncated and hence bounded L\'evy measure. Essentially it is an
extension of the method in \cite{Cifani/Jakobsen} from symmetric $\alp$-stable
to general L\'evy measures, but since non-symmetric measures are
allowed, the discretization becomes  more complicated here. 
Apart from its   
ability to capture the correct solution for the whole family of
equations of the form \eqref{1}, the main advantage of our numerical
method is that it allows for a complete error analysis through the new
framework for error estimates that we develop in the second part of
the paper.

The second, and probably most important contribution of the paper, is
the development of a theory capable of producing error estimates for
degenerate equations of order greater than 1. This theory is based on
a non-trivial extension of the Kuznetsov theory 
for scalar conservation laws \cite{Kuznetsov} to the current
fractional diffusion setting. An initial step in this analysis was
performed in \cite{Alibaud/Cifani/Jakobsen}, with the derivation
of a so-called Kuznetsov lemma in a relevant form for \eqref{1}. In
\cite{Alibaud/Cifani/Jakobsen} the lemma is used in the derivation of
continuous dependence estimates and error estimates for
vanishing viscosity type of approximations of \eqref{1}. In the
present paper, we show how it can be used in solving the more difficult
problem of finding error estimates
for numerical methods for \eqref{1}.

As a corollary of our Kuznetsov type theory, we obtain explicit
$\lambda$-dependent error estimates 
when $\mu$ is a measure satisfying
\begin{equation}\label{fractional_meas}
\begin{split}
0\leq \mathbf 1_{|z|<1}d\mu(z)\leq c_\lambda\frac{dz}{|z|^{d+\lambda}}\qquad\text{for}\qquad c_\lambda>0\text{
and }\lambda\in(0,2).
\end{split}
\end{equation}
In this paper we will call such measures {\em fractional measures}.
For example for the implicit version of our numerical method
\eqref{scheme_implicit}, we prove in Section \ref{sec:framework} that
$$\|u(\cdot,T)-u_{\dx}(\cdot,T)\|_{L^1(\R^d)} \leq C_T\left\{
\begin{array}{ll}
\dx^\frac{1}{2}&\lambda\in(0,1),\\
\dx^\frac{1}{2}\log(\dx)&\lambda=1,\\
\dx^\frac{2-\lambda}{2}&\lambda\in(1,2),
\end{array}
\right.$$
where $u$ is the entropy solution of \eqref{1} and $u_{\dx}$ is the
solution of \eqref{scheme_implicit}. Note that our error estimate
covers all values $\lambda\in(0,2)$, all spacial dimensions $d$, and
possibly strongly degenerate equations!
Also note that under our assumptions, the 
solution $u$ possibly only have BV regularity in space. Hence the
error estimate is 
robust in the sense that it holds also for discontinuous
solutions, and moreover, the classical result of
Kuznetsov \cite{Kuznetsov}
for conservation laws follows as a corollary by taking $A\equiv0$ (a
valid choice here!) and
$\lambda\in(0,1)$. The above estimate is also consistent with 
error estimates for  the vanishing $\lambda$-fractional viscosity method,
$$\partial_tu+\Div f(u)=-\dx\,(-\Delta)^{\lambda/2}u\quad\text{as}\quad\dx\ra0^+,$$
see e.g. \cite{DrVV,Alibaud}, but note that our problem is
different and much more difficult.

There is a vast literature on approximation schemes and error
estimates for scalar conservations laws, we refer e.g. to the books
\cite{Kr:Book,Holden/Risebro} and references therein for more details.
For local degenerate convection-diffusion equations like
\eqref{locCD}, some approximation methods and error estimates can be
found e.g. in \cite{EK00,EGH02,Karlsen/Koley/Risebro} and references
therein. In this setting it is very difficult to obtain error
estimates for numerical methods, and the only result we are aware of is a
very recent one by Karlsen et al. \cite{Karlsen/Koley/Risebro} (but
see also \cite{Chen/Karlsen}). This very nice result applies to rather
general equations of the form \eqref{locCD} but in one space dimension
and under additional regularity assumptions (e.g. $\del_x(A(u))\in
BV$). When it comes to nonlocal 
convection-diffusion equations, the literature is very recent and
not yet very extensive. The paper 
\cite{DeRo04} introduce finite volume schemes for radiation
hydrodynamics equations, a model where $\Levy^\mu$ is a
nonlocal derivative of order $0$. Then fractional conservation laws
are discretized in \cite{Droniou,Cifani/Jakobsen/Karlsen,CJ10} with finite
difference, discontinuous Galerkin, and spectral vanishing viscosity
methods respectively. In \cite{DeRo04,Cifani/Jakobsen/Karlsen}
Kuznetsov type error estimates are given, but only for integrable L\'evy
measures or measures like \eqref{frac_lap} with $\lambda<1$. Both of
these results can be obtained through the framework of this paper. In
\cite{CJ10} error estimates 
are given for all $\lambda$ but with completely different
methods. The general degenerate non-linear case is discretized in
\cite{Cifani/Jakobsen} (without error estimates) for symmetric
$\alp$-stable L\'evy measures and then in the most general case in the
present paper.  

Linear non-degenerate versions of \eqref{1} frequently arise in
Finance, and the problem of solving these equations numerically has generated
a lot of activity over the last decade. An introduction and overview
of this activity can be found in the book
\cite{Cont/Tankov}, including numerical schemes based on
truncation of the L\'evy measure. We also mention the literature on
fractional and nonlocal fully non-linear equations like e.g. the Bellman
equation of optimal control theory. Such equations have been
intensively studied over the last decade using viscosity solution
methods, including initial results on numerical methods and error
analysis. We refer e.g. \cite{BI08,Biswas,JKLC08} and references therein for an
overview and the most general results in that direction. In fact,
ideas from that field has been essential in the development of the
entropy solution theory of equations like \eqref{1}, and the
construction of monotone numerical methods of this paper parallels the
one in \cite{Biswas}. However the structure of the two classes of
equations along with their mathematical and numerical analysis are
very different.  

This paper is organized as follows. In Section
\ref{sec:prelim}  we recall the
entropy formulation and well-posedness results for \eqref{1} of
\cite{Cifani/Jakobsen} and the Kuznetsov type lemma derived in
\cite{Alibaud/Cifani/Jakobsen}. We present the numerical method in
Section \ref{sec:num_met}. There we focus on the case of no convection
($f\equiv 0$) to simplify the exposition and focus on new ideas. In Section 
\ref{sec:properties} we prove several auxiliary properties of the
numerical method which will be useful in the following sections. We
establish existence, uniqueness, and a priori estimates for the
solutions of the numerical method in Section \ref{sec:comp}. The general 
Kuznetsov type theory for deriving error estimates is presented in 
Section \ref{sec:framework}, where it is also used to establish a rate of
convergence for equations with 
fractional L\'evy measures, i.e. \eqref{fractional_meas} holds. 
In Section \ref{sec:convection_equations} we extend all the
results considered so far to general convection-diffusion equations of the
form \eqref{1} with $f\not\equiv0$. Finally, we give the proof of the
main error estimate Theorem \ref{th:kuz_prior} in Section \ref{sec:pf}.


\section{Preliminaries}
\label{sec:prelim}
In this section we briefly recall the entropy formulation for equations of
the form \eqref{1} introduced in \cite{Cifani/Jakobsen}, and the new
Kuznetsov type of lemma established in \cite{Alibaud/Cifani/Jakobsen}. 
 Let $\eta(u,k)=|u-k|$, $\eta'(u,k)=\sgn(u-k)$,
$q_l(u,k)=\eta'(u,k)\,(f_l(u)-f_l(k))$ for $l=1,\ldots,d$, and
write the nonlocal operator $\Levy^\mu[\phi]$ as
$$\Levy_{r}^\mu[\phi]+\Levy^{\mu,r}[\phi]+
\gamma^{\mu,r}\cdot\nabla\phi,$$
 where
\begin{equation*}
\begin{split}
\Levy_{r}^\mu[\phi](x)&=\int_{0<|z|\leq r}\phi(x+z)-\phi(x)-z\cdot\nabla\phi(x)\textbf{1}_{|z|\leq1}\ \dif\mu(z),\\
\Levy^{\mu,r}[\phi](x)&=\int_{|z|>r}\phi(x+z)-\phi(x)\ \dif\mu(z),\\
\gamma^{\mu,r}_l&=-\int_{|z|>r}z_l\textbf{1}_{|z|\leq1}\
\dif\mu(z),\qquad l=1,\ldots,d.
\end{split}
\end{equation*}
We also define $\mu^\ast$ by $\mu^\ast(B)=\mu(-B)$ for
all Borel sets $B\not\ni0$. Let us recall that
\begin{align*}
\int_{\R^d}\varphi(x)\,\Levy^{\mu}[\psi](x)\ \dif x
=\int_{\R^d}\psi(x)\,\Levy^{\mu^\ast}[\varphi](x)\ \dif x
\end{align*}
for all smooth $L^\infty\cap L^1$ functions $\varphi,\psi$,
cf.~\cite{Alibaud/Cifani/Jakobsen,Cifani/Jakobsen}.

\begin{definition}\emph{(Entropy solutions)}\label{def:entropy}
A function $u\in L^\infty(Q_T)\cap C([0,T];L^1(\R^d))$ is an
entropy solution of \eqref{1} if, for all $k\in\mathbb{R}$,  $r>0$, and test
functions $0\leq\varphi\in C^\infty_c(\R^d\times[0,T])$,
\begin{equation}\label{entropy_ineq}
\begin{split}
&\int_{Q_{T}}\eta(u,k)\,\partial_t\varphi+\left(q(u,k)+\gamma^{\mu^\ast,r}\right)\cdot\nabla\varphi+\eta(A(u),A(k))\,\Levy^{\mu^\ast}_{r}[\varphi]\\
&\qquad\qquad\qquad\qquad\qquad\qquad\qquad\qquad\qquad+\eta'(u,k)\,\Levy^{\mu,r}[A(u)]\,\varphi\ \dif x\, \dif t\\
&\qquad-\int_{\R^d}\eta(u(x,T),k)\,\varphi(x,T)\ \dif
x+\int_{\R^d}\eta(u_0(x),k)\,\varphi(x,0)\ \dif x\geq 0.
\end{split}
\end{equation}
\end{definition}

Note that $\gamma^{\mu,r}_l\equiv
0$ when the L\'{e}vy measure $\mu$ 
is symmetric, i.e. when $\mu^*\equiv\mu$. From \cite{Cifani/Jakobsen}
we now have the following well-posedness result.
\begin{theorem}\emph{(Well-posedness)}
\label{thm:WP}
Assume (A.1) -- (A.4) hold. Then there exists a unique entropy solution $u$ of  \eqref{1} such that 
$$u\in L^\infty(Q_T)\cap
C([0,T];L^1(\R^d))\cap L^\infty(0,T;BV(\mathbb{R}^d)),$$
and the following a priori estimates hold
\begin{align*}
&\|u(\cdot,t)\|_{L^\infty(\R^d)}\leq\|u_0\|_{L^\infty(\mathbb{R}^d)},\\
&\|u(\cdot,t)\|_{L^1(\R^d))}\leq\|u_0\|_{L^1(\mathbb{R}^d)},\\
&|u(\cdot,t)|_{BV(\mathbb{R}^d)}\leq|u_0|_{BV(\mathbb{R}^d)},\\
&\|u(\cdot,t)-u(\cdot,s)\|_{L^1(\R^d)}\leq \sigma(|t-s|),
\end{align*}
for all $t,s\in[0,T]$ where 
$$\sigma(r)=\begin{cases}c\, r&\text{if } \int_{|z|>0}|z|\wedge 1\ \dif\mu(z)<\infty,\\
c\, r^{\frac1{2}}&\text{otherwise.}
\end{cases}$$
Moreover, if also \eqref{fractional_meas} holds, then
$$\sigma(r)=\begin{cases}c\, r&\text{if }\lambda\in(0,1),\\
c\, |r\ln r|&\text{if }\lambda=1,\\
c\, r^{\frac1{\lambda}}&\text{if }\lambda\in(1,2).\end{cases}$$
\end{theorem}
The last a priori estimate is slightly more general then the one in
\cite{Cifani/Jakobsen}, and follows e.g. in the limit from the
estimates in Lemmas \ref{lem:time-reg} and \ref{lem:time-reg_frac}.
We now recall the new Kuznetsov type of
lemma established in \cite{Alibaud/Cifani/Jakobsen}.
Let
\begin{equation*}
\begin{split}
\omega\in C_{c}^{\infty}(\mathbb{R}),\quad
0\leq\omega\leq1,\quad\text{$\omega(\tau)=0$ for all
$|\tau|>1$,}\quad\text{and}\quad \int_{\mathbb{R}}\omega(\tau)\,\dif
\tau=1,
\end{split}
\end{equation*}
and define
$\omega_{\delta}(\tau)=\frac{1}{\delta}\,\omega\left(\frac{\tau}{\delta}\right)$,
$\Omega_\epsilon(x)=\omega_\epsilon(x_1)\cdots\omega_\epsilon(x_d)$, and
$$\varphi^{\epsilon,\delta}(x,y,t,s)=\Omega_{\epsilon}(x-y)\,\omega_\delta(t-s)$$
for $\epsilon,\delta >0$. We also need
\begin{equation}\label{time_mod}
\begin{split}
\mathcal{E}_\delta(v)=\sup_{\substack{|t-s|<\delta\\
t,s\in[0,T]}}\|v(\cdot,t)-v(\cdot,s)\|_{L^1(\R^d)}.
\end{split}
\end{equation}
In the following we let $\dif w=\dif x\, \dif t\,
\dif y\,\dif s$ and $C_T\geq0$ be a constant depending on time and the 
initial data $u_0$ that may change from line to line. 

\begin{lemma}\label{lem:kuznetsov}\emph{(Kuznetsov type of lemma)}
Assume (A.1) -- (A.4) hold. Let $u$ be the entropy solution of
\eqref{1} and $v$ be any function in 
$L^\infty(Q_{T})\cap
C([0,T];L^1(\R^d))\cap L^\infty(0,T;BV(\R^d)) $ with $v(\cdot\,,0)=v_0(\cdot)$.
Then, for any $\epsilon,r>0$ and $0<\delta<T$,
\begin{equation*}
\begin{split}
&\|u(\cdot,T)-v(\cdot,T)\|_{L^{1}(\mathbb{R}^d)}\leq\|u_0-v_0\|_{L^{1}(\mathbb{R}^d)}+C\,(\epsilon+\mathcal{E}_\delta(u)\vee\mathcal{E}_\delta(v))\\
&\quad-\iint_{Q_{T}}\iint_{Q_{T}}\eta(v(x,t),u(y,s))\,\partial_t\varphi^{\epsilon,\delta}(x,y,t,s)\ \dif w\\
&\quad-\iint_{Q_{T}}\iint_{Q_{T}}q(v(x,t),u(y,s))\cdot\nabla_x\varphi^{\epsilon,\delta}(x,y,t,s)\ \dif w\\
&\quad+\iint_{Q_{T}}\iint_{Q_{T}}\eta(A(v(x,t)),A(u(y,s)))\,\Levy_{r}^{\mu^\ast}[\varphi^{\epsilon,\delta}(x,\cdot,t,s)](y)\ \dif w\\
&\quad-\iint_{Q_{T}}\iint_{Q_{T}}\eta'(v(x,t),u(y,s))\,\Levy^{\mu,r}[A(v(\cdot,t))](x)\,\varphi^{\epsilon,\delta}(x,y,t,s)\ \dif w\\
&\quad-\iint_{Q_{T}}\iint_{Q_{T}}\eta(A(v(x,t)),A(u(y,s)))\,\gamma^{\mu^\ast,r}\cdot\nabla_x\varphi^{\epsilon,\delta}(x,y,t,s)\
\dif w\\
&\quad+\iint_{Q_T}\int_{\R^d} \eta(v(x,T),u(y,s))\,\varphi^{\epsilon,\delta}(x,T,y,s) \ \dif x \, \dif y \, \dif s\\
&\quad-\iint_{Q_T}\int_{\R^d}
\eta(v_0(x),u(y,s))\,\varphi^{\epsilon,\delta}(x,0,y,s) \ \dif x \,
\dif y \, \dif s
\end{split}
\end{equation*}
\end{lemma}

The proof is given in \cite{Alibaud/Cifani/Jakobsen}. The original
result of result of Kuznetsov in \cite{Kuznetsov} is a special case when
$\mu=0$ (or $A=0$).


\section{The numerical method}\label{sec:num_met}

In this section we derive our numerical method.
Here and in the following sections we focus on the case $f\equiv0$ to simplify
the exposition and focus on the new ideas. The general case $f\neq 0$
will then be treated at the end, in Section \ref{sec:convection_equations}. 

We will consider uniform space/time grids given by
$x_\alpha=\alpha\,\Delta x$ for $\alpha\in\mathbb{Z}^d$ and
$t_n=n\,\Delta t$ for $n=0,\ldots,N=\frac{T}{\Delta t}$. We also use
the following rectangular subdivisions of space
$$R_\alpha=x_\alpha+\Delta x\,(0,1)^d \quad\text{for}\quad \alpha\in\mathbb{Z}^d.$$

We start by discretizing the nonlocal operator, replacing the measure
$\mu$ by the bounded truncated measure 
  $\mathbf{1}_{|z|>\frac\dx2}(z)\mu$ and the gradient by a numerical
  gradient
\begin{align}\label{mag8}
\hat D_{\Delta x}=(\hat D_{1},\cdots,\hat D_{d}),
\end{align}
where $\hat D_{l}\equiv D^\gamma_l$ are upwind finite difference
operators defined by
\begin{equation}\label{diff_operator}
 D^\gamma_l\phi(x)=\left\{
\begin{split}
D^+_l\phi(x):=\frac{\phi(x+\dx\: e_l)-\phi(x)}{\Delta
  x}&\qquad\text{for $\gamma_l^{\mu,\frac{\Delta x}{2}}>0$},\\ 
D^-_l\phi(x):=\frac{\phi(x)-\phi(x-\dx\: e_l)}{\Delta x}&\qquad\text{otherwise}.
\end{split}
\right.
\end{equation}
Here $e_1,\dots,e_d$ is the standard basis of $\R^d$.
This gives an approximate nonlocal operator 
\begin{equation}\label{newB}
\begin{split}
&\hat\Levy^\mu[A(\phi)](x)\\
&=\int_{|z|>\frac{\Delta x}{2}}A(\phi(x+z))-A(\phi(x))\,\dif\mu(z) +\gamma^{\mu,\frac{\Delta x}{2}}\cdot \hat
D_{\Delta x}A(\phi(x)),
\end{split}
\end{equation}
which is monotone by upwinding and non-singular since the truncated
measure is bounded.  

A semidiscrete approximation of \eqref{1} with $f\equiv0$ is then
obtained by solving the approximate equation 
\begin{align}
\label{appr1}
\partial_tu=\hat\Levy^\mu[A(u)],
\end{align}
by a finite volume method on the spacial subdivision
$\{R_\alp\}_\alp$. I.e. for each $t$, we look for piecewise constant
approximate solution
$$U(x,t)=\sum_{\beta\in\Z^d}U_\beta(t)\,\mathbf{1}_{R_\beta}(x),$$
that satisfy \eqref{appr1} in weak form with
$\frac1{\dx^d}\mathbf{1}_{R_\beta}$ as test functions: For every $\alp\in\Z^d$,
\begin{align*}
\frac1{\dx^d}\int_{R_\alp}\partial_tU\, dx =\frac1{\dx^d}\int_{R_\alp}\hat\Levy^\mu[A(U)]\, dx.
\end{align*}

Finally we discretize in time by replacing $\del_t$ by  backward or
forward differences $D^\pm_{\Delta t}$ and $U_\alp(t)$ by a piecewise constant
approximation $U^n_\alp$. The result is the implicit method
\begin{align}
U_{\alpha}^{n+1}&=U_{\alpha}^{n}+\Delta t\, \hat\Levy^\mu\langle
A(U^{n+1})\rangle_\alpha,\label{scheme_implicit}
\end{align}
and the explicit method
\begin{align}
U_{\alpha}^{n+1}&=U_{\alpha}^{n}+\Delta t\, \hat\Levy^\mu\langle
A(U^{n})\rangle_\alpha\label{scheme_explicit}
\end{align}
where
\begin{equation*}\label{dis_non_loc}
\begin{split}
\hat\Levy^\mu\langle A(U^n)\rangle_\alpha
&=\frac{1}{\Delta x^d}\int_{R_{\alpha}}\hat\Levy^\mu[A(\bar
U^n)](x)\ \dif x,
\end{split}
\end{equation*}
and $\bar
U^n(x)=\sum_{\beta\in\Z^d}U_\beta^n\,\mathbf{1}_{R_\beta}(x)$ is a
piecewise constant $x$-interpolation of $U$. As initial condition for
both methods we take
\begin{equation*}
\begin{split}
U_{\alpha}^{0}=\frac{1}{\Delta x^d}\int_{R_{\alpha}}u_{0}(x)\ \dif
x \qquad\text{for all}\qquad \alpha\in\Z^d.
\end{split}
\end{equation*}

\begin{lemma}
\label{lemLG}
\begin{equation*}
\begin{split}
\hat\Levy^\mu\langle A(U^n)\rangle_\alpha=\sum_{\beta\in\Z}G_\beta^\alpha\,A(U^n_\beta)
\end{split}
\end{equation*}
with
$G_\beta^\alpha=G_{\alpha,\beta}+G^{\alpha,\beta}$ and
\begin{equation}\label{weights}
\begin{split}
G_{\alpha,\beta}&=\frac{1}{\Delta
x^d}\int_{R_{\alpha}}\int_{|z|>\frac{\Delta x}{2}}\mathbf{1}_{R_\beta}(x+z)-\mathbf{1}_{R_\beta}(x)\ \dif\mu(z)\,\dif x,\\
G^{\alpha,\beta}&=\sum_{l=1}^d\gamma_l^{\mu,\frac{\Delta
x}{2}}\frac{1}{\Delta
x^d}\int_{R_{\alpha}}D^\gamma_l\mathbf{1}_{R_\beta}(x)\ \dif x.
\end{split}
\end{equation}
\end{lemma}

\begin{remark}
 $G_{\alpha,\beta}$ is a Toeplitz matrix
 (cf. Lemma \ref{lem:properties} (b)) while
$G^{\alpha,\beta}$ is a tridiagonal matrix. When the
measure $\mu$ is symmetric, then  $G_{\alpha,\beta}$ is symmetric and
$G^{\alpha,\beta}=0$. 
\end{remark}
\begin{proof}
Since
$$A(\bar
U(x))=\sum_{\beta\in\Z^d}A(U_\beta)\,\mathbf{1}_{R_\beta}(x)\qquad\text{and}\qquad\hat 
D_{\dx}\bar U(x)=\sum_{\beta\in\Z^d} 
U_\beta\,\hat D_{\dx}\mathbf{1}_{R_\beta}(x),$$
we find that
\begin{align*}
&\Delta x^d\hat\Levy^\mu\langle A(U)\rangle_\alpha
=\int_{R_{\alpha}}\hat\Levy^\mu[A(\bar
U)](x)\ \dif x\\
&=\int_{R_\alpha}\int_{|z|>\frac{\Delta x}{2}}A(\bar U(x+z))-A(\bar U(x))\ \dif\mu(z)\,\dif x\\
&\qquad\qquad\qquad\qquad+\gamma^{\mu,\frac{\Delta x}{2}}\cdot\int_{R_\alpha}
\sum_{\beta\in\Z^d}A(U_\beta)\,\hat D_{\dx}\mathbf{1}_{R_\beta}(x)\ \dif x\\
&=\sum_{\beta\in\Z^d}A(U_\beta)\left(\int_{R_{\alpha}}
\int_{|z|>\frac{\Delta x}{2}}\mathbf{1}_{R_\beta}(x+z)-\mathbf{1}_{R_\beta}(x)\ 
\dif\mu(z)\,\dif x\right)\\
&\qquad\qquad\qquad\qquad+\sum_{\beta\in\Z^d}A(U_\beta)
\left(\sum_{l=1}^d\gamma_l^{\mu,\frac{\Delta x}{2}}\int_{R_\alpha}D_l^\gamma
\mathbf{1}_{R_\beta}(x)\ \dif x\right).
\end{align*}
The proof is complete.
\end{proof}


\section{Properties of the numerical method}\label{sec:properties}

In this section we show that the numerical methods are conservative,
monotone and consistent in the sense that certain cell entropy
inequalities are satisfied. We start by a technical lemma summarizing
the properties of the weights $G_\beta^\alpha$ defined in
\eqref{weights}. 

\begin{lemma}\ \label{lem:properties}

\noindent (a)
$\sum_{\alpha\in\Z^d}G_\beta^\alpha=\sum_{\alpha\in\Z^d}G_\alpha^\beta=0$
for all $\beta\in\Z^d$.
\smallskip

\noindent (b) $G_\alpha^\beta=G_{\alpha+e_l}^{\beta+e_l}$ for all
$\alpha,\beta\in\Z^d$ and $l=1,\ldots,d$.
\smallskip

\noindent (c) $G_\beta^\beta\leq0$ and
$G_\beta^\alpha\geq 0$ for $\alpha\neq \beta$.
\smallskip 

\noindent (d) There is $\bar c=\bar c(d,\mu)>0$ such that
$G_\beta^\beta\geq-\frac{\bar c}{\hat\sigma_\mu(\dx)}$ and 
where
\begin{align}
\label{sigma0}
\hat\sigma_{\mu}(s)=\begin{cases}
s&\text{when }\int|z|\wedge1\, d\mu(z)<\infty,\\
s^{2}&\text{otherwise}.\end{cases}
\end{align}
\noindent (e)  If \eqref{fractional_meas} holds, then there is $\bar
c=\bar c(d,\lambda)>0$ such that $G_\beta^\beta\geq
-\frac{\bar c}{\hat\sigma_\lambda(\dx)}$ for
\begin{align}
\label{sigma}
\hat\sigma_\lambda(s)=\begin{cases}s^{\lambda}&\text{for }\lambda> 1,\\
\frac{s}{|\ln s|}&\text{for }\lambda= 1,\\
s&\text{for }\lambda< 1.\end{cases}
\end{align}
\end{lemma}

\begin{proof} (a) By the definitions of
$G_{\alpha,\beta},G^{\alpha,\beta}$ and Fubini's theorem, 
\begin{align*}
\Delta x^d\sum_{\alpha\in\Z^d}G_{\alpha,\beta}&=\int_{|z|>\frac{\Delta
    x}{2}}\Big(\int_{\R^d}\mathbf{1}_{R_\beta}(x+z)\ \dif
x-\int_{\R^d}\mathbf{1}_{R_\beta}(x)\ \dif x\Big)\,\dif\mu(z)=0,\\
\Delta
x^d\sum_{\alpha\in\Z^d}G^{\alpha,\beta}&=\pm\sum_{l=1}^d\frac{\gamma_l^{\mu,\frac{\Delta 
x}{2}}}{\dx}\Big(\int_{\R^d}\mathbf{1}_{R_\beta}(x\pm\dx\:e_l)\dif x-\int_{\R^d}\mathbf{1}_{R_\beta}(x)\dif x\Big)=0,
\end{align*}
and, since $\sum_{\beta\in\Z^d}\mathbf{1}_{R_\beta}(x)\equiv 1$,
\begin{align*}
\Delta
x^d\sum_{\beta\in\Z^d}G_{\alpha,\beta}&=\int_{R_\alpha}\int_{|z|>\frac{\Delta
x}{2}}\Big(\sum_{\beta\in\Z^d}\mathbf{1}_{R_\beta}(x+z)-\sum_{\beta\in\Z^d}\mathbf{1}_{R_\beta}(x)\Big)\
\dif\mu(z)\,\dif x=0,\\
\Delta
x^d\sum_{\beta\in\Z^d}G^{\alpha,\beta}&=\pm\sum_{l=1}^d\frac{\gamma_l^{\mu,\frac{\Delta 
x}{2}}}{\dx}\int_{R_\alpha}\Big(\sum_{\beta\in\Z^d}\mathbf{1}_{R_\beta}(x\pm\dx\:e_l)-\sum_{\beta\in\Z^d}\mathbf{1}_{R_\beta}(x)\Big)\dif x=0.
\end{align*}
Therefore
$\sum_{\alpha\in\Z^d}G_\beta^\alpha=\sum_{\alpha\in\Z^d}\left(G_{\alpha,\beta}+G^{\alpha,\beta}\right)=0$ 
and $\sum_{\beta\in\Z^d}G_\beta^\alpha=0$.

\medskip

\noindent (b) Let $y=x+e_l$ and note that
\begin{equation*}
\begin{split}
\Delta
x^d\ G_{\alpha,\beta}
&=\int_{R_{\alpha+e_l}}\int_{|z|>\frac{\Delta
x}{2}}\mathbf{1}_{R_\beta}(y-\dx\, e_l+z)-\mathbf{1}_{R_\beta}(y-\dx\,  e_l)\
\dif\mu(z)\,\dif y\\
&=\int_{R_{\alpha+e_l}}\int_{|z|>\frac{\Delta
x}{2}}\mathbf{1}_{R_{\beta+e_l}}(y+z)-\mathbf{1}_{R_{\beta+e_l}}(y)\
\dif\mu(z)\,\dif y\\
&=\Delta
x^d\ G_{\beta+e_l,\alpha+e_l}.
\end{split}
\end{equation*}
In a similar fashion we get $G^{\alpha,\beta}=G^{\beta+e_l,\alpha+e_l}$.

\medskip

\noindent (c) 
Note that
\begin{align*}
\Delta
x^d\ G_{\beta,\beta}&=\int_{R_{\beta}}\int_{|z|>\frac{\Delta
x}{2}}\mathbf{1}_{R_{\beta}}(x+z)-1\ \dif\mu(z)\ \dif x\leq 0.\\
\end{align*}
while by the definition $D^\gamma_l$, see \eqref{diff_operator},
\begin{align*}
\Delta
x^d\ G^{\beta,\beta}
&=-\sum_{l=1}^d\gamma_l^{\mu,\frac{\Delta
x}{2}}\,\mathrm{sgn}\left(\gamma_l^{\mu,\frac{\Delta
x}{2}}\right)\,\int_{R_\beta}\frac{\mathbf{1}_{R_\beta}(x)}{\Delta x}\ \dif x\leq 0.
\end{align*}
For $\alpha\neq \beta$,
\begin{equation*}
\begin{split}
\Delta
x^d\ G_{\alpha,\beta}&=\int_{R_{\alpha}}\int_{|z|>\frac{\Delta
x}{2}}\mathbf{1}_{R_{\beta}}(x+z)\ \dif\mu(z)\, \dif x\geq 0,
\end{split}
\end{equation*}
$G^{\alpha,\beta}=0$ for $\alpha\neq\beta\pm e_l$, and by the
definition of $D^\gamma_l$, 
\begin{equation*}
\Delta
x^d\ G^{\beta\pm e_l,\beta}=
\sum_{l=1}^d\gamma_l^{\mu,\frac{\Delta
x}{2}}\,\mathrm{sgn}\left(\gamma_l^{\mu,\frac{\Delta
x}{2}}\right)\,\int_{R_\beta\pm e_l}\frac{\mathbf{1}_{R_\beta}(x\pm\dx\:e_l)}{\Delta
x}\ \dif x\geq 0.
\end{equation*}
Therefore $G_\beta^\beta=G_{\beta,\beta}+G^{\beta,\beta}\leq 0$ and
$G_\beta^\alpha=G_{\alpha,\beta}+G^{\alpha,\beta}\geq 0$ for $\alpha\neq \beta$.
\smallskip

\noindent (d) To find the lower bound on  $G^\beta_\beta$ we note that
$\int_{R_{\beta}}\mathbf{1}_{R_{\beta}}(x+z)-\mathbf{1}_{R_{\beta}}(x)\
\dif x\geq -\dx^d$, 
and hence 
\begin{align*}
 G_{\beta,\beta}&\geq-\int_{|z|>\frac{\Delta
x}{2}} \ \dif\mu(z)\geq
-\int_{|z|<1}\bigg(\bigg(\frac{|z|}{\frac{\dx}{2}}\bigg)^2\mathbf{1}_{|z|<1}(z)+
\mathbf{1}_{|z|>1}(z)\bigg)\ \dif \mu(z).
\end{align*}
The bound then follows since
$$\dx\, G^{\beta,\beta}\geq -d\int_{\frac{\dx}2<|z|<1}|z|\ \dif\mu(z)\geq -d\int_{0<|z|<1}\frac{|z|^2}{\frac{\dx}2}\ \dif\mu(z).$$
When $\int|z|\wedge1\,d\mu(z)<\infty$, the corresponding bound follows by a similar argument.
\smallskip

\noindent (e) When \eqref{fractional_meas} hold we can estimate
$G_{\beta,\beta}$ in the following way
\begin{align*}
 G_{\beta,\beta}&\geq
-\int_{\frac\dx2<|z|<1}\frac{|z|}{\frac{\dx}{2}}\frac{c_\lambda\dif
  z}{|z|^{d+\lambda}}-\int_{|z|>1}\dif \mu(z)\\
&= -\begin{cases}
c_\lambda\frac{2}{\dx}\frac{\sigma_d}{1-\lambda}\Big(1-\big(\frac{\dx}{2}\big)^{1-\lambda}\Big)+C&\text{for
} \lambda\neq 1,\\[0.2cm]
-c_\lambda\frac{2}{\dx}\sigma_d\ln\frac{\dx}{2} +C& \text{for
} \lambda= 1.
\end{cases}
\end{align*}
The last equality can be proved using polar coordinates, and $\sigma_d$ is
the surface area of the unit sphere in $\R^d$. Similarly we find that
$$\dx\, G^{\beta,\beta}\geq -d\int_{\frac{\dx}2<|z|<1}|z|\ \frac{c_\lambda\dif z}{|z|^{d+\lambda}}= -dc_\lambda\begin{cases}
\frac{\sigma_d}{1-\lambda}\Big(1-\big(\frac{\dx}{2}\big)^{1-\lambda}\Big)&\text{for
} \lambda\neq 1,\\[0.2cm]
-\sigma_d\ln\frac{\dx}{2} & \text{for
} \lambda= 1,
\end{cases}$$
and since $\Big(1-(\frac{\dx}{2})^{1-\lambda}\Big)$
is less than $1$ or
$(\frac{\dx}{2})^{1-\lambda}$ when $\lambda<1$ or
$\lambda>1$ respectively (and when $\dx<2$), the proof is complete.
\end{proof}

From the two facts that $G_\alpha^\beta\geq0$ when $\alpha\neq\beta$ and
$\mathrm{sgn}(u)A(u)=|A(u)|$, we now immediately get a Kato type
inequality for the discrete nonlocal operator \eqref{dis_non_loc}. 
\begin{lemma}\label{lem:app1}(Discrete Kato inequality)
If $\left\{u_\alpha,v_\alpha\right\}_{\alpha\in\Z^d}$ are two bounded
sequences, then
$$\mathrm{sgn}(u_\alpha-v_\alpha)\sum_{\beta\in\Z^d}G_\beta^\alpha(A(u_\beta)-A(v_\beta))\leq\sum_{\beta\in\Z^d}G_\beta^\alpha\left|A(u_\beta)-A(v_\beta)\right|.$$
\end{lemma}

From Lemma \ref{lem:properties} it also follows that the explicit method
\eqref{scheme_explicit} and the implicit method
\eqref{scheme_implicit} are conservative and monotone, at least when the
explicit method satisfies the following CFL condition:
\begin{align}
\bar c L_A\,\frac{\Delta t}{\hat\sigma_\mu(\Delta
  x)}<1\qquad\text{where $\hat\sigma_\mu$
is defined in \eqref{sigma0}}.
\label{cfl}
\end{align}
Here $\bar c$ is defined in Lemma \ref{lem:properties}, and $L_A$ denotes the Lipschitz constant of $A$. When the L\'{e}vy
measure $\mu$ also satisfies \eqref{fractional_meas}, we have a weaker
CFL condition
\begin{equation}\label{CFLj}
\begin{split}
\bar cL_A\frac{\Delta t}{\hat\sigma_\lambda(\Delta
  x)}<1\qquad\text{where $\hat\sigma_\lambda$ 
is defined in \eqref{sigma}}.
\end{split}
\end{equation}

\begin{proposition}[Conservative monotone schemes]\ 
\label{prop:CM}

\noindent (a) The implicit and explicit methods \eqref{scheme_implicit} and
\eqref{scheme_explicit} are conservative, i.e. for an $l^1$-solution $U$,
$$\sum_\alp U^{n}_\alp = \sum_\alp U^{0}_\alp.$$

\noindent (b) The implicit method is monotone, i.e. if $U$
and $V$ solve \eqref{scheme_implicit}, then
$$U^n\leq V^n \qquad\Rightarrow\qquad U^{n+1}\leq V^{n+1}
\quad\text{for}\quad n\geq0.$$ 

\noindent (c) If \eqref{cfl} (or \eqref{CFLj} and
\eqref{fractional_meas}) holds, then the explicit method
\eqref{scheme_explicit} is monotone. 
\end{proposition}
\begin{remark} 
The CFL condition \eqref{cfl} implies that $\frac{\dt}{\dx^2}\leq C$
in general (just as for the heat equation), and
$\frac{\dt}{\dx}\leq C$ when $\int|z|\wedge1\,d\mu(z)<\infty$. Condition \eqref{cfl} is sufficient for all 
equations considered in this paper. In real applications however,
typically \eqref{fractional_meas} holds, and the superior
CFL condition \eqref{CFLj} should be used.
\end{remark}

\begin{proof}
(a) Sum \eqref{scheme_implicit} or \eqref{scheme_explicit} over
$\alpha$, change the order 
of summation, and use Lemma \ref{lem:properties} (a):
\begin{equation*}
\begin{split}
\sum_{\alpha\in\Z^d}U^{n+1}_\alpha&=\sum_{\alpha\in\Z^d}U^{n}_\alpha+\dt\sum_{\beta\in\Z^d}A(U_\beta)\bigg(\sum_{\alpha\in\Z^d}G_\beta^\alpha\bigg)
=\sum_{\alpha\in\Z^d}U^{n}_\alpha.
\end{split}
\end{equation*}

\noindent (c) Let
$T_\alpha[u]=u_\alpha+\Delta
t\sum_{\beta\in\Z^d}G_\beta^\alpha\,A(u_{\beta})$, the right hand side
of \eqref{scheme_explicit}. By Lemma \ref{lem:properties}
(c), $G^\alpha_\beta\geq0$ for $\alpha\neq\beta$ and hence
\begin{equation*}
\partial_{u_\beta}T_\alpha[u]\geq 0\quad\text{for}\quad\beta\neq \alpha.
\end{equation*}
Since $A$ non-decreasing and $G_\alpha^\alpha\leq 0$, we use the lower
bound on $G_\alpha^\alpha$ in Lemma \ref{lem:properties} (c) to find that
\begin{equation*}
\begin{split}
&\partial_{u_\alpha}T_\alpha[u]=1+\Delta
t\,G_\alpha^\alpha\,A'(u_{\alpha})\geq 1-\bar c L_A\frac{\dt}{\hat\sigma_\mu(\dx)},
\end{split}
\end{equation*}
which is positive by the CFL condition \eqref{cfl}.

\smallskip
\noindent (b) The proof is similar to and easier than the proof of (c).
\end{proof}

We then turn to checking the
consistency of the method, and to do that we write
$G_{\alpha,\beta}=G_{\alpha,\beta}^r+G_{\alpha,\beta,r}$ and 
$G^{\alpha,\beta}=G^{\alpha,\beta,r}+G^{\alpha,\beta}_r$
for $r>0$ where
\begin{equation*}
\begin{split}
G_{\alpha,\beta}^r&=\frac1{\dx^d}\int_{R_{\alpha}}\int_{\frac{\Delta x}{2}<|z|\leq r}\mathbf{1}_{R_\beta}(x+z)-\mathbf{1}_{R_\beta}(x)\ \dif\mu(z)\,\dif x,\\
G_{\alpha,\beta,r}&=\frac1{\dx^d}\int_{R_{\alpha}}\int_{|z|>
r}\mathbf{1}_{R_\beta}(x+z)-\mathbf{1}_{R_\beta}(x)\
\dif\mu(z)\,\dif x,\\
G^{\alpha,\beta,r}&=\frac1{\dx^d}\sum_{l=1}^d\gamma_{l,r}^{\mu,\frac{\Delta
x}{2}}\int_{R_{\alpha}}D^{\gamma_r}_l\mathbf{1}_{R_\beta}(x)\ \dif x\\
&\text{for}\quad\gamma^{\mu,\frac{\Delta x}{2}}_{l,r}=-\int_{\frac{\Delta
x}{2}<|z|\leq r}z_l\textbf{1}_{|z|\leq1}\ \dif\mu(z),\\
G^{\alpha,\beta}_r&=\frac1{\dx^d}\sum_{l=1}^d\gamma_{l}^{\mu,r}\int_{R_{\alpha}}D^{\gamma^r}_l\mathbf{1}_{R_\beta}(x)\
\dif x.
\end{split}
\end{equation*}
If $r<\frac\dx2$, we set $G^r_{\alp,\beta}=0=G^{\alp,\beta,r}$. We
also define
$$G_\alpha^{\beta,r}=G_{\alpha,\beta}^r+G^{\alpha,\beta,r}\quad\text{and}\quad
G_{\alpha,r}^\beta=G_{\alpha,\beta,r}+G^{\alpha,\beta}_r,$$
and note that Lemmas
\ref{lem:properties} and \ref{lem:app1} obviously still holds with
$G_\alpha^{\beta,r}$ or 
$G_{\alpha,r}^\beta$ replacing $G_{\alpha}^\beta$. 

\begin{proposition}\label{th:conv}\emph{(Cell-entropy inequalities)}

\noindent (a) If $U$ is a solution of the implicit method
\eqref{scheme_implicit}, then, for all $r>0$ and $k\in\R$, 
\begin{equation}\label{j4}
\begin{split}
\eta(U_{\alpha}^{n+1},k)\leq\eta(U_{\alpha}^{n},k)
&+\Delta t\sum_{\beta\in\Z^d}G_\beta^{\alpha,r}\,\eta(A(U^{n+1}_{\beta}),A(k))\\
&+\Delta t\,
\eta'(U^{n+1}_\alpha,k)\sum_{\beta\in\Z^d}G_{\beta,r}^\alpha\,A(U^{n+1}_{\beta}).
\end{split}
\end{equation}

\noindent (b) Assume the CFL condition \eqref{cfl} (or \eqref{CFLj} and
\eqref{fractional_meas}) holds. If $U$ is a solution of the explicit method
\eqref{scheme_explicit}, then, for all $r>0$ and $k\in\R$,
\begin{equation}\label{u1}
\begin{split}
\eta(U_{\alpha}^{n+1},k)\leq\eta(U_{\alpha}^{n},k)
&+\Delta t\sum_{\beta\in\Z^d}G_\beta^{\alpha,r}\,\eta(A(U^{n}_{\alpha}),A(k))\\
&+\Delta t\,
\eta'(U^{n+1}_\alpha,k)\sum_{\beta\in\Z^d}G_{\beta,r}^\alpha\,A(U^{n}_{\alpha}).
\end{split}
\end{equation}
\end{proposition}

\begin{remark}
In the cell-entropy inequality for the explicit method, the
$\eta'$-term appears in the ``wrong'' time. In Section 
\ref{sec:framework}, we will see that this leads to worse error
estimates for the explicit method than for the implicit method.
\end{remark}

\begin{remark}[Convergence to entropy solutions]
\label{conv_rem}
Proposition \ref{th:conv} and a standard argument show that any
$C([0,T];L^1_{\mathrm{loc}}(\R^d))$-convergent sequence of
(interpolated) solutions $\bar u_{\dx}$ of \eqref{scheme_implicit} or
\eqref{scheme_explicit}, will converge to an entropy solution of
\eqref{1}. We refer to Theorem 3.9 in
\cite{Holden/Risebro} and Section 4.2 in \cite{Cifani/Jakobsen} for
more details. Convergence to the entropy solution also follows from the
error estimates of Section \ref{sec:framework}.
\end{remark}

\begin{proof}
\noindent (a) By \eqref{scheme_implicit} we easily see that for any $k\in\R$,
\begin{align*}
&U_{\alpha}^{n+1}\vee k\leq U_{\alpha}^{n}\vee k+\Delta t\,
\mathbf{1}_{(k,+\infty)}(U_{\alpha}^{n+1})\,\hat\Levy^\mu\langle
A(U^{n+1})\rangle_\alpha,\\
&U_{\alpha}^{n+1}\wedge k\geq U_{\alpha}^{n}\wedge k+\Delta t\,
\mathbf{1}_{(-\infty,k)}(U_{\alpha}^{n+1})\,\hat\Levy^\mu\langle
A(U^{n+1})\rangle_\alpha.
\end{align*}
Subtracting and using $\eta(u,k)=|u-k|$  and
$\eta'(u,k)=\sgn(u-k)$, we find that 
\begin{equation*}
\begin{split}
\eta(U_{\alpha}^{n+1},k)\leq\eta(U_{\alpha}^{n},k)+\Delta
t\,\eta'(U^{n+1}_\alpha,k)\,\hat\Levy^\mu\langle
A(U^{n+1})\rangle_\alpha.
\end{split}
\end{equation*}
For any $r>0$, we use Lemmas \ref{lem:properties} (a) and \ref{lem:app1} with
$G_\beta^{\alpha,r}$ replacing $G_\beta^{\alpha}$ to see that
\begin{equation*}
\begin{split}
&\eta'(U^{n+1}_\alpha,k)\sum_{\beta\in\Z^d}G_\beta^{\alpha,r}\,A(U^{n+1}_{\beta})\\
&=\eta'(U^{n+1}_\alpha,k)\sum_{\beta\in\Z^d}G_\beta^{\alpha,r}\,(A(U^{n+1}_{\beta})-A(k))\qquad\left(\text{since }\sum_{\beta\in\Z^d}G_\beta^{\alpha,r}=0\right)\\
&\leq\sum_{\beta\in\Z^d}
G_\beta^{\alpha,r}\,\eta(A(U^{n+1}_{\beta}),A(k)).
\end{split}
\end{equation*}
The cell entropy inequality now follows from writing
$G^\alp_\beta=G^\alp_{\beta,r}+G^{\alp,r}_\beta$ and using the  above
inequalities.
\medskip

\noindent (b) By \eqref{scheme_explicit} and monotonicity (Proposition
\ref{th:conv} (c)) we obtain the following inequalities: For all 
$r>0$,
\begin{align*}
U_{\alpha}^{n+1}\vee k\leq U_{\alpha}^{n}\vee k
&+\Delta t\sum_{\beta\in\Z^d}G_\beta^{\alpha,r}\,A(U^{n}_{\beta}\vee k)\\
&+\Delta t\,
\mathbf{1}_{(k,+\infty)}(U_{\alpha}^{n+1})\sum_{\beta\in\Z^d}G_{\beta,r}^\alpha\,A(U^{n}_{\beta}),\\
U_{\alpha}^{n+1}\wedge k\geq U_{\alpha}^{n}\wedge k
&+\Delta t\sum_{\beta\in\Z^d}G_\beta^{\alpha,r}\,A(U^{n}_{\beta}\wedge k)\\
&+\Delta t\,
\mathbf{1}_{(-\infty,k)}(U_{\alpha}^{n+1})\sum_{\beta\in\Z^d}G^\alpha_{\beta,r}\,A(U^{n}_{\beta}).
\end{align*}
Since $\eta(A(U),A(k))=A(U\vee k)-A(U\wedge k)$, the cell entropy
inequality follows from subtracting the two inequalities.
\end{proof}


\section{A priori estimates, existence, and uniqueness}\label{sec:comp}

In this section we state and prove several a priori estimates for the
solutions of the numerical methods \eqref{scheme_implicit} and
\eqref{scheme_explicit}. 
In what follows, we will use different interpolants $\bar
u$ of the solutions $U^n_\alp$ of the schemes. For the implicit method
\eqref{scheme_implicit} we take
\begin{equation}\label{interp_impl}
\begin{split}
\bar u(x,t)=U_{\alpha}^{n+1}\quad\text{for all $(x,t)\in
R_{\alpha}\times(t_n,t_{n+1}]$},
\end{split}
\end{equation}
while for the explicit method \eqref{scheme_explicit},
\begin{equation}\label{interp_expl}
\begin{split}
\bar u(x,t)=U_{\alpha}^{n}\quad\text{for all $(x,t)\in
R_{\alpha}\times[t_n,t_{n+1})$}.
\end{split}
\end{equation}

We now prove the following a priori estimates for $\bar u$:
\begin{align}
\|\bar u(\cdot,t)\|_{L^{1}(\mathbb{R}^d)}&\leq\|u_{0}\|_{L^{1}(\mathbb{R}^d)},\label{L1-cont}\\
\|\bar u(\cdot,t)\|_{L^{\infty}(\mathbb{R}^d)}&\leq\|u_{0}\|_{L^{\infty}(\mathbb{R}^d)},\label{max-princ}\\
|\bar
u(\cdot,t)|_{BV(\mathbb{R}^d)}&\leq|u_{0}|_{BV(\mathbb{R}^d)}.\label{BV-cont}
\end{align}

\begin{lemma}\label{th:existence_scheme}\emph{(A priori estimates)}

\noindent(a) If $U$ solve \eqref{scheme_implicit} and $\bar u$ is defined
by \eqref{interp_impl}, then the a priori estimates \eqref{L1-cont}
-- \eqref{BV-cont} hold for all $t>0$. 
\smallskip

\noindent(b) Assume the CFL condition \eqref{cfl} (or \eqref{CFLj} and
\eqref{fractional_meas}) holds. If $U$ solve \eqref{scheme_explicit} and $\bar u$ is defined
by \eqref{interp_expl}, then the a priori estimates \eqref{L1-cont}
-- \eqref{BV-cont} hold for all $t>0$.
\end{lemma}

\begin{proof}
Since the schemes are conservative and monotone, cf. Proposition
\ref{prop:CM}, this is a standard result that essentially follows
from the Crandall-Tartar Lemma. For explicit methods in part (b) we refer to 
e.g. Theorem 3.6 in \cite{Holden/Risebro} for the details. 

We did not find a reference for implicit methods, so we give a proof
of part (a) here. See also \cite{Droniou} for the case when $A$ is
linear. Let $u_\alpha=U_\alpha^{n+1}$, 
$h_\alpha=U_{\alpha}^{n}$, and write \eqref{scheme_implicit} as 
\begin{equation}\label{h1}
u_{\alpha}-\Delta t\sum_{\beta\in\Z^d}G_\beta^\alpha\,A(u_{\beta})=h_\alpha.
\end{equation}

We prove \eqref{L1-cont}. Multiply  \eqref{h1} by
$\mathrm{sgn}(u_\alpha)$ and use Lemma \ref{lem:app1} to get 
\begin{equation*}
\begin{split}
|u_{\alpha}|-\Delta t\sum_{\beta\in\Z^d}G_\beta^\alpha\,|A(u_{\beta})|\leq|h_\alpha|,
\end{split}
\end{equation*}
which by Fubini's theorem and the fact that
$\sum_{\alpha\in\Z^d}G_\beta^\alpha=0$ implies that
\begin{align*}
\sum_{\alpha\in\mathbb{Z}^d}|u_{\alpha}|\leq\sum_{\alpha\in\mathbb{Z}^d}|h_\alpha|.
\end{align*}
By the definition of $u_\alp,h_\alp$ and an iteration in $n$, it follows
that
$$\sum_{\alpha\in\mathbb{Z}^d}|U^n_{\alpha}|\leq\sum_{\alpha\in\mathbb{Z}^d}|U^0_\alpha|.$$
By \eqref{interp_impl}, $\|\bar
u(\cdot,t)\|_{L^1(\R^d)}=\dx^d\sum_{\alpha\in\mathbb{Z}^d}|U^n_{\alpha}|$
for $t\in (t_n,t_{n+1}]$, and \eqref{L1-cont} follows.
\medskip

To prove \eqref{BV-cont}, we subtract two equations \eqref{h1}
evaluated at different points, 
\begin{equation*}
\begin{split}
u_{\alpha}-u_{\alpha-e_l}-\Delta
t\sum_{\beta\in\Z^d}\Big(G_\beta^\alpha\,A(u_{\beta})-G_\beta^{\alpha-e_l}\,A(u_{\beta})\Big)= 
h_\alpha-h_{\alpha-e_l}
\end{split}
\end{equation*}
and use the fact that $G_\beta^\alpha=G_{\beta+e_l}^{\alpha+e_l}$ to see that
\begin{equation*}
u_{\alpha}-u_{\alpha-e_l}-\Delta t\sum_{\beta\in\Z^d}G_\beta^\alpha\,\Big(A(u_{\beta})-A(u_{\beta-e_l})\Big)=
h_\alpha-h_{\alpha-e_l}.
\end{equation*}
Then we multiply by $\mathrm{sgn}(u_{\alpha}-u_{\alpha-e_l})$, use Lemma
\ref{lem:app1}, and sum over $\alpha$, to find that
\begin{equation*}
\begin{split}
\sum_{\alpha\in\mathbb{Z}^d}\left|u_{\alpha}-u_{\alpha-e_l}\right|\leq\sum_{\alpha\in\mathbb{Z}^d}\left|h_\alpha-h_{\alpha-e_l}\right|.
\end{split}
\end{equation*}
The estimate \eqref{BV-cont} then follows by iteration and the
definitions of $u_\alp,h_\alp,\bar u$.
\medskip

It remains to prove \eqref{max-princ}. Note that since
$\sum_\alp|u_\alp|<\infty$ by \eqref{L1-cont}, there is an $\alp_0$
such that $\sup_\alp u_\alp=u_{\alp_0}$. Moreover, the parabolic term is
nonpositive at the maximum point: since
$\sum_{\beta\in\Z^d}G_\beta^\alpha=0$ and $\sum_{\beta\in\Z^d}|G^\alpha_\beta|<\infty$, 
$$\sum_{\beta\in\Z^d}G_\beta^{\alp_0}\,A(u_\beta)=\sum_{\beta\in\Z^d}G_\beta^{\alp_0}\Big(A(u_\beta)-A(u_{\alp_0})\Big)\leq
0.$$
Then by the above inequality and \eqref{h1}, 
\begin{equation*}
\sup_{\alpha\in\mathbb{Z}^d}u_\alp=u_{\alpha_0}\leq u_{\alp_0}-\Delta
t\sum_{\beta\in\Z^d}G_\beta^{\alp_0}\,A(u_{\beta})=h_{\alp_0}\leq\sup_{\alpha\in\mathbb{Z}^d}h_\alpha.  
\end{equation*}
In a similar way we find that
$\inf_{\alpha\in\mathbb{Z}^d}h_\alpha\leq\inf_{\alpha\in\mathbb{Z}^d}u_\alpha$
and \eqref{max-princ} follow from the definitions of
$u_\alp,h_\alp,\bar u$ and an iteration in $n$.
\end{proof}

\begin{lemma}[Global existence and uniqueness]\ 

\noindent (a) There exists a unique solution $U^n\in l^1$ of the
implicit scheme \eqref{scheme_implicit} for all $n\geq0$.
\smallskip

\noindent (b) Assume the CFL condition \eqref{cfl} (or \eqref{CFLj} and
\eqref{fractional_meas}) holds. Then there
exists a unique solution $U^n\in l^1$ of the
explicit scheme \eqref{scheme_explicit} for all $n\geq0$.
\end{lemma}

Note that $U^n\in l^1$ implies that $\bar u(\cdot,t)\in L^1(\R^d)$.

\begin{proof}
(a) Let $u_\alpha=U_\alpha^{n+1}$ and $h_\alpha=U_{\alpha}^{n}$, 
rewrite \eqref{scheme_implicit} as \eqref{h1}, define
\begin{equation*}
T_\alpha[u]=u_\alpha-\epsilon\bigg(u_{\alpha}-\Delta
  t\sum_{\beta\in\Z^d}G_\beta^\alpha\,A(u_{\beta})-h_\alpha\bigg), 
\end{equation*}
and let $\epsilon$ be such that
\begin{align*}
\epsilon\bigg(1+L_A\bar c\frac{\Delta t}{\hat\sigma_\mu(\Delta
x)}\bigg)< 1.
\end{align*}

We first show that $T_\alpha$ is monotone, i.e. $u\leq
v$ implies $T_\alpha[u]\leq T_\alpha[v]$. For $\alpha\neq\beta$,
$G^\alpha_\beta\geq0$ by Lemma \ref{lem:properties}, and hence since
$A$ non-decreasing, 
\begin{equation*}
\partial_{u_\beta}T_\alpha[u]\geq 0.
\end{equation*}
Moreover, since $A$ non-decreasing and $-\frac{\bar
  c}{\hat\sigma_\mu(\dx)}\leq G_\alpha^\alpha\leq 0$, 
\begin{equation*}
\begin{split}
\partial_{u_\alpha}T_\alpha[u]&=1-\epsilon+\epsilon\,\Delta
t\,G_\alpha^\alpha\,A'(u_{\alpha})\geq 1-\epsilon\bigg(1+L_A\bar c\frac{\Delta t}{\hat\sigma_\mu(\Delta
x)}\bigg)
\end{split}
\end{equation*}
which is positive by our choice of $\epsilon$.

Since $T$ is monotone and $A$ is nondecreasing,
\begin{align*}
&\sum_\alp\Big(T_\alpha[u]-T_\alpha[v]\Big)^+\leq\sum_\alp\Big( T_\alpha[u\vee v]-T_\alpha[v]\Big)\\
&=(1-\epsilon)\,\sum_{\alp\in\Z^d} (u_\alpha\vee v_\alp-v_\alpha)+\epsilon\,\Delta
t\sum_{\alp\in\Z^d} \sum_{\beta\in\Z^d}G_\beta^\alpha\,\Big(A(u_\beta\vee v_\beta)-A(v_\beta)\Big)\\
&= (1-\epsilon)\,\sum_{\alp\in\Z^d} (u_\alpha-v_\alpha)^++\epsilon\,\Delta
t\sum_{\beta\in\Z^d}\Big(\sum_{\alp\in\Z^d} G_\beta^\alpha\Big)\Big(A(u_\beta)-A(v_\beta)\Big)^+.
\end{align*}
A similar estimate holds for $\sum_\alp(T_\alpha[u]-T_\alpha[v])^-$,
and since $\sum_{\alp\in\Z^d} G_\beta^\alpha=0$, we have shown that
\begin{equation*}
\begin{split}
\sum_{\alpha\in\mathbb{Z}^d}|T_\alpha[u]-T_\alpha[v]|\leq(1-\epsilon)\sum_{\alpha\in\mathbb{Z}^d}|u_\alpha-v_\alpha|.
\end{split}
\end{equation*}
So $T_\alp$ is an $l^1$-contraction and  Banach's
fixed point theorem then implies that there exists a unique solution $\bar
u\in l^1$ of $T_\alpha[\bar u]=\bar u_\alpha$ and hence also of
\eqref{h1}.
\bigskip

\noindent(b) Existence follows by construction and the a priori
  estimates in Lemma \ref{th:existence_scheme}. Uniqueness essentially
  follows by monotonicity and $\sum_\alp G^\alp_\beta=0$: Assume two
  solutions $U^n$ and $V^n$, subtract the two equations and multiply by
  $\sgn(U^n-V^n)$, and use the Kato inequality (Lemma \ref{lem:app1}) along with
  $\sum_\alp G^\alp_\beta=0$ to show that 
  $\sum_\alp|U^n-V^n|\leq\sum_\alp|U^0-V^0|$. 
\end{proof}

We have the following regularity estimate in time:
\begin{lemma}\label{lem:time-reg}\emph{(Regularity in time)}

\noindent (a) Assume (A.2) -- (A.4) hold, and let $U$ be a solution of
the implicit method \eqref{scheme_implicit} and $\bar u$ defined by
\eqref{interp_impl}. Then 
$$\|\bar u(\cdot,s)-\bar
u(\cdot,t)\|_{L^{1}(\mathbb{R}^d)}\leq \sigma_{\mu}(|s-t|+\dt)$$ 
for all $s,t>0$, where 
\begin{equation*}
\sigma_{\mu}(r)=\begin{cases}
r&\qquad\text{if $\int_{|z|>0}|z|\wedge 1\ \dif\mu(z)<\infty$,}\\[0.2cm]
\sqrt{r}&\qquad\text{otherwise.}
\end{cases}
\end{equation*}
\smallskip

\noindent (b) Assume (A.2) -- (A.4) and \eqref{cfl} (or \eqref{CFLj}
and \eqref{fractional_meas}) hold, and let $U$ be a solution of
the explicit method \eqref{scheme_explicit} and $\bar u$ defined by
\eqref{interp_expl}. Then 
$$\|\bar u(\cdot,s)-\bar
u(\cdot,t)\|_{L^{1}(\mathbb{R}^d)}\leq \sigma_{\mu}(|s-t|+\dt)$$ 
for all $s,t>0$, where $\sigma_\mu$ is defined in (a).
\end{lemma}

\begin{proof}
The two proofs are essentially identical, so we only do the proof for case (a).
\medskip

\noindent 1) By \eqref{scheme_implicit}, we find that for any $x\in R_\alp$,
\begin{align*}
U^n_\alp - U^{n-1}_\alp= \frac{\dt}{\dx^d} \int_{R_\alp}\hat\Levy
[A(\bar U^n)](x)\,dx.
\end{align*}
Take a test function $0\leq\phi\in C_c^\infty$ and define
$\phi_\alp=\frac1{\dx^d} \int_{R_\alp}\phi(y)\dif y$ and $\bar
\phi(x)=\sum_\alp \phi_\alp\mathbf{1}_{R_\alp}(x)$. Multiply the
equation by $\dx^d\phi_\alp$ and sum over $\alp$ to find 
that
\begin{align*}
&\int_{\R^d}\bar\phi(x)(\bar U^{n}(x)-\bar U^{n-1}(x))\ \dif x=\dt\int_{\R^d}\bar\phi(x)\hat\Levy
[A(\bar U^n)](x)\, \dif x,
\end{align*}
where 
Let $\hat\Levy^\ast$ be the adjoint of $\hat\Levy$, then since $\bar
U$ is constant over $R_\alp$,
\begin{align*}
&\int_{\R^d}\phi(x)(\bar U^{n}(x)-\bar U^{n-1}(x))\ \dif
x=\int_{\R^d}\bar \phi(x)(\bar U^{n}(x)-\bar U^{n-1}(x))\ \dif x\\
&=\dt\int_{\R^d}(\bar\phi(x)-\phi(x))\hat\Levy
[A(\bar U^n)](x)\, \dif x+\dt\int_{\R^d}\hat\Levy^\ast
[\phi](x)A(\bar U^n)(x)\, \dif x.
\end{align*}
\smallskip

\noindent 2) Let  $\omega_\eps$ be an approximate unit, i.e.
$\omega_\eps(x)=\frac1{\eps^d}\omega(\frac x \eps)$ where $0\leq
\omega\in C_0^\infty$ and $\int_{\R^d}\omega \,dx=1$. Take
$\phi(x)=\omega_\eps(y-x)$ in the equation above and let $U^n_\eps=\bar U^n\ast \omega_\eps$:
\begin{align*}
&\bar U^n_\eps - \bar U^{n-1}_\eps=\dt(\bar\omega_\eps-\omega_\eps)\ast\hat\Levy
[A(\bar U^n)]+\dt\hat\Levy^\ast
[\omega_\eps]\ast A(\bar U^n).
\end{align*}
By Fubini we then find that
\begin{align*}
&\frac1\dt\|\bar U^n_\eps - \bar U^{n-1}_\eps\|_{L^1}\leq\|\bar\omega_\eps-\omega_\eps\|_{L^1}\|\hat\Levy
[A(\bar U^n)]\|_{L^1}+\|\hat\Levy^\ast
[\omega_\eps]\ast A(\bar U^n)\|_{L^1}=I_1+I_2.
\end{align*}
\smallskip

\noindent3) To estimate $I_1$, note that by a standard argument
$$\|\bar\omega_\eps-\omega_\eps\|_{L^1}\leq
|\omega_\eps|_{BV}\dx=\frac{c_\omega}{\eps}\dx,$$
and then by the definition of $\hat\Levy$ in \eqref{newB}, Fubini, the $L^1\cap BV$
regularity of $U^n$ (Lemma \ref{th:existence_scheme}), and the
regularity of $A$ in (A.2),
\begin{align*}
&\|\hat\Levy[A(\bar U^n)]\|_{L^1}\\
&=\int_{|z|>\frac\dx2}\int_{\R^d}A(\bar
U^n(x+z))-A(\bar U^n(x))-z\cdot \hat D_{\dx} A(\bar
U^n(x))\mathbf{1}_{|z|<1 } \ \dif x\, \dif\mu(z)\\
&\leq
\int_{|z|>\frac\dx2}\Big(2|A(U^n)|_{BV}|z|\mathbf{1}_{|z|<1}+2\|A(U^n)\|_{L^1}\mathbf{1}_{|z|>1}\Big)\,\dif\mu(z)
\\
&\leq C \int_{|z|>\frac\dx2}|z|\wedge 1\,\dif\mu(z)\leq \frac C{\dx}\int_{|z|>0}|z|^2\wedge 1\,\dif\mu(z).
\end{align*} 
These estimates along with (A.3) shows that $I_1\leq C\eps^{-1}$.
\medskip

\noindent 4) Then we estimate $I_2$. Note first that since $\hat
D_{\dx}=D+(\hat D_{\dx}-D)$, we can use Taylor's formula to see that
\begin{align*}
&\phi(x+z)-\phi(x)-z\hat D_{\dx}
\phi(x)\\
&=\int_0^1(1-s)z^TD^2\phi(x+sz)z\,\dif s \pm
\dx
\sum_{i=1}^dz_i\int_0^1(1-s)\phi_{x_ix_i}(x\pm s\dx)\,\dif s. 
\end{align*}
This identity along with the definition of $\hat\Levy^\ast$, repeated
use of Fubini, and one integration by parts in $x$, then leads to
\begin{align*}
&\hat\Levy^\ast[\omega_\eps]\ast A(\bar U^n)(x)\\
&=-\int_0^1\int_{\frac\dx2<|z|<1}\int_{\R^d}(1-s)D\omega_\eps(x-y+sz)z\otimes
z\, DA(\bar U^n(y))\,\dif y\,\dif\mu(z)\,\dif s\\
&\quad \mp\dx\sum_{i=1}^d\int_0^1\int_{\frac\dx2<|z|<1}\int_{\R^d}(1-s)
\del_{x_i}\omega_\eps(x-y\pm s\dx)z_i\,
\del_{x_i}A(\bar U^n(y))\,\dif y\dif\mu(z)\dif s\\
&\quad +\int_{\R^d}\int_{|z|>1}\Big(\omega_\eps(x-y+z)-\omega_\eps(x-y)\Big)A(\bar U^n(x)) \,\dif\mu(z)\,\dif y.
\end{align*}
Here  $DA(\bar U^n(y))\,\dif y$ should be interpreted as a measure, and
 $\int|DA(\bar U^n(y))|\,\dif y=\int \dif|A(U^n)|(y)
 =|A(U^n)|_{BV}$. By Young's inequality for convolutions (Fubini in our
 case), we then find that
\begin{align*}
I_2&\leq \frac32|\omega_\eps|_{BV}|A(\bar
U^n)|_{BV}\int_{0<|z|<1}|z|^2\,\dif\mu(z)
+2\|\omega_\eps\|_{L^1}\|A(\bar
U^n)\|_{L^1}\int_{|z|>1}\dif\mu(z)\\
&\leq C\eps^{-1}.
\end{align*} 
Here again we have used the properties and regularity of $\mu$, $A$,
$\bar U^n$, and $\omega_\eps$.
\medskip

\noindent 5) By steps 2) -- 4) we can conclude that
$$\|\bar U^n_\eps - \bar U^{m}_\eps\|_{L^1}\leq \sum_{j=m+1}^n\|\bar U^j_\eps - \bar U^{j-1}_\eps\|_{L^1}\leq \frac C\eps |n-m|\dt,$$
where the constant $C$ does not depend on $n$ or $m$. By the triangle
inequality and standard $BV$-estimates, it then follows that
\begin{align*}
\|\bar U^n - \bar U^{m}\|_{L^1}&\leq \|\bar U^n - \bar
U^{n}_\eps\|_{L^1}+\|\bar U^n_\eps - \bar U^{m}_\eps\|_{L^1}+\|\bar
U^m_\eps - \bar U^{m}\|_{L^1}\\
&\leq |\bar U^n|_{BV}\eps+\frac C\eps |n-m|\dt +|\bar
U^m|_{BV}\eps,
\end{align*}
and hence by taking $\eps=C\sqrt{|n-m|\dt}$,
\begin{align*}
\|\bar U^n - \bar U^{m}\|_{L^1}
&\leq C\sqrt{|n-m|\dt}.
\end{align*}
For the time-interpolated function $\bar u$ defined in
\eqref{interp_impl}, we then find the following estimate
\begin{align*}
\|\bar u(\cdot,t) - \bar u(\cdot,s)\|_{L^1}&= \|\bar U^n - \bar
U^{m}\|_{L^1}
\leq C\sqrt{|n-m|\dt}\leq C\sqrt{|t-s|+\dt}.
\end{align*}
The equality follows since for each $t,s$ there are $n,m$ such that
$\bar u(x,t)=\bar U^n(x)$ and $\bar u(x,s)=\bar U^m(x)$. Moreover, by
the definition of $\bar u$, $ |n-m|\dt\leq|t-s| +\dt$.
\medskip

It remains to prove a better estimate for the case when $\int
|z|\wedge1\,\dif\mu(z)<\infty$. This proof is similar but much easier than
the proof above, so we skip it.
\end{proof}

The time regularity result in Lemma \ref{lem:time-reg} is not optimal
for Levy operators $\Levy$ with order in the interval $[1,2)$. To get
optimal results we need more detailed information on the Levy measure
$\mu$ than merely assumption (A.3). We will now prove an improved
time regularity result for fractional measures \eqref{fractional_meas}. 
In this result we will need the following CFL condition,
\begin{equation}\label{CFLjj}
\begin{split}
 C\frac{\Delta t}{\Delta
  x^{1\vee\lambda}}<1\qquad\text{for}\qquad \lambda\in(0,2).
\end{split}
\end{equation}

\begin{lemma}\label{lem:time-reg_frac}\emph{(Time regularity for
    fractional measures)} Assume (A.2) -- (A.4), and
  \eqref{fractional_meas} hold. 
\smallskip

\noindent (a) If the CFL condition \eqref{CFLjj} hold, $U$ is a
solution of the implicit method \eqref{scheme_implicit} and $\bar u$
its interpolation defined by \eqref{interp_impl}, then for all $s,t>0$,
$$\|\bar u(\cdot,s)-\bar
u(\cdot,t)\|_{L^{1}(\mathbb{R}^d)}\leq
\sigma_{\lambda}(|s-t|+\dt);\qquad\sigma_\lambda(\tau)=\left\{ 
\begin{array}{ll}
\tau&\lambda< 1,\\
\tau|\ln\tau|&\lambda=1,\\
\tau^{\frac1{\lambda}}&\lambda> 1.
\end{array}
\right.$$

\noindent (b) If the CFL condition \eqref{CFLj} hold, $U$ is a
solution of the explicit method \eqref{scheme_explicit} and $\bar u$
its interpolation defined by \eqref{interp_expl}, then for all $s,t>0$,
$$\|\bar u(\cdot,s)-\bar
u(\cdot,t)\|_{L^{1}(\mathbb{R}^d)}\leq
\sigma_{\lambda}(|s-t|+\dt);\quad \sigma_\lambda(\tau)=\left\{
\begin{array}{ll}
\tau&\lambda< 1,\\
\tau^\alp\ \text{for any $\alp\in(0,1)$}&\lambda=1,\\
\tau^{\frac1{\lambda}}&\lambda> 1.
\end{array}
\right.$$ 
\end{lemma}

Note well that in this result we {\em need} the CLF condition also for
the implicit scheme. The reason is that the time-regularity is linked
through the equation to the approximate $\dx$-depending diffusion term
as will be seen from the proof. For the implicit scheme, we can have
better results for $\lambda=1$ since we can use the less restrictive
CFL condition \eqref{CFLjj}.

\begin{proof}
The result for $\lambda<1$ is a corollary to Lemma \ref{lem:time-reg}.
The proof for $\lambda\geq1$ is the same as the proof of Lemma
\ref{lem:time-reg}, except that we use different estimates for $I_1$
and $I_2$ in step 2).  From step 
3) in that proof and \eqref{fractional_meas} and a simple
computation in polar coordinates, we get that 
\begin{align*}
I_1&\leq C\frac{\dx}{\eps}\int_{|z|>\frac{\dx}2}|z|\wedge
1\,\dif\mu(z)\leq
C\frac{\dx}{\eps}\bigg(\int_{\frac{\dx}2<|z|<1}\frac{|z|\,
dz}{|z|^{d+\lambda}}+C\bigg)\\
&\leq \frac C\eps \begin{cases}\dx+\dx^{2-\lambda}&\text{for }\lambda>1,\\
 \dx-\dx\ln \dx&\text{for }\lambda=1.
 \end{cases}
\end{align*}
To estimate $I_2$, we use Taylor expansions and integration by parts
to find that 
\begin{align*}
&\hat\Levy^\ast[\omega_\eps]\ast A(\bar U^n)(x)=\\
&-\int_0^1\int_{\frac\dx2<|z|<\eps}\int_{\R^d}(1-s)D\omega_\eps(x-y+sz)z\otimes
z\, DA(\bar U^n(y))\,\dif y\,\dif\mu(z)\,\dif s\\
&-\int_0^1\int_{\eps<|z|<1}\int_{\R^d}\Big(\omega_\eps(x-y+sz)-\omega_\eps(x-y)\Big)z\, DA(\bar U^n(y))\,\dif y\,\dif\mu(z)\,\dif s\\
& \mp\dx\sum_{i=1}^d\int_0^1\int_{\frac\dx2<|z|<\eps}\int_{\R^d}(1-s)
\del_{x_i}\omega_\eps(x-y\pm s\dx)z_i\,
\del_{x_i}A(\bar U^n(y))\,\dif y\dif\mu(z)\dif s\\
&\mp\sum_{i=1}^d\int_0^1\int_{\eps<|z|<1}\int_{\R^d}\Big(\omega_\eps(x-y\pm s\dx)-\omega_\eps(x-y)\Big)
z_i\,\del_{x_i}A(\bar U^n(y))\,\dif y\dif\mu(z)\dif s\\
&+\int_{\R^d}\int_{|z|>1}\Big(\omega_\eps(x-y+z)-\omega_\eps(x-y)\Big)A(\bar U^n(x)) \,\dif\mu(z)\,\dif y.
\end{align*}
Then by Fubini, the definition of $\omega_\eps$, and the change of variables
$(x,z)\ra(\eps x,\eps z)$, 
$$\int_{\R^d}\int_{\frac\dx2<|z|<\eps} |D\omega_\eps(x+sz)| |z|^2\,\dif\mu(z)
\leq c_\lambda\eps^{1-\lambda}\int_{\R^d} |D\omega|\, \dif x \int_{0<|z|<1}
\frac{|z|^2\,\dif z}{|z|^{d+\lambda}}.$$ 
By similar estimates and Young's inequality for convolutions we find that
\begin{align*}
I_2
\leq\  &c_\lambda\eps^{1-\lambda}|A(\bar
U^n)|_{BV}\bigg(3|\omega|_{BV}\int_{0<|z|<1}\frac{|z|^2\,\dif z}{|z|^{d+\lambda}}
+4\|\omega\|_{L^1}\int_{1<|z|<\frac1\eps}\frac{|z|\,\dif
  z}{|z|^{d+\lambda}}\bigg)\\
&+2\|A(\bar
U^n)\|_{L^1}\|\omega_\eps\|_{L^1}\int_{|z|>1}\dif \mu(z)\\
\leq\  &C\begin{cases}\eps^{1-\lambda}+1,& \lambda>1,\\
|\ln\eps|+1,& \lambda=1.
\end{cases}
\end{align*} 
Note that the $\ln\eps$-term comes from the integral over $1<|z|<\frac
1\eps$. 

As in step 5) in the proof of Lemma \ref{lem:time-reg}, we then find that
\begin{align*}
\|\bar U^n - \bar U^{m}\|_{L^1}&\leq |\bar U^n|_{BV}\eps+|n-m|\dt(I_1+I_2) +|\bar
U^m|_{BV}\eps.
\end{align*}
To conclude, we assume that $\dx\leq \eps$ which means in
particular that 
$$I_1+I_2\leq C\begin{cases}\eps^{1-\lambda}+1, &
  \lambda>1,\\|\ln\eps|+1,& \lambda=1.
\end{cases}$$
When $\lambda>1$, the final result follows from taking
$\eps=c(|n-m|\dt)^{\frac1\lambda}$
and arguing as in the end of the proof of Lemma
\ref{lem:time-reg}. Note that in view of the CFL conditions
\eqref{CFLj} and \eqref{CFLjj}, the constant $c$ can be chosen such
that $\dx\leq\eps$. 
For $\lambda=1$, we can use $\eps=c|n-m|\dt$ for the implicit method in
view of \eqref{CFLjj}, and by \eqref{CFLj}, $\eps=c(|n-m|\dt)^{\alp'}$
for any $\alp'\in(0,1)$, will do the job for the explicit method.
\end{proof}

By the a priori estimates Lemma \ref{th:existence_scheme}
and \ref{lem:time-reg} and Kolmogorov's compactness theorem
(cf.~e.g.~\cite[Theorem 3.8]{Holden/Risebro}), we find subsequences of
both methods \eqref{scheme_implicit} and \eqref{scheme_explicit}
converging to some function $u$. The function $u$ inherits all the a
priori estimates of $\bar u$, and it will be the unique entropy
solution of \eqref{1} by Remark \ref{conv_rem}. In short, we 
have the following result:
\begin{theorem}\emph{(Compactness)}\label{th:compactness}
Assume (A.2) -- (A.4) hold. If either
\begin{itemize}
\item[(i)]  $U$ is the solution of
the implicit method \eqref{scheme_implicit} and $\bar u$ defined by
\eqref{interp_impl}, or
\item[(ii)] $U$ is the solution of
the explicit method \eqref{scheme_explicit}, $\bar u$ defined by
\eqref{interp_expl}, and \eqref{cfl} (or \eqref{CFLj}
and \eqref{fractional_meas}) also holds, 
\end{itemize}
then there is a subsequence of $\{\bar u\}_{\Delta
x>0}$ converging in $C([0,T];L^1(\R^d))$ to the unique entropy
solution $u$ of \eqref{1} as $\Delta x\rightarrow 0$. Moreover, 
$$u\in L^\infty(Q_T)\cap C([0,T];L^{1}(\mathbb{R}^d))\cap
L^\infty(0,T;BV(\mathbb{R}^d)).$$
\end{theorem}

\begin{remark}
 This result provides a proof for the existence result Theorem 5.3 in
\cite{Cifani/Jakobsen} 
for $L^1\cap L^\infty\cap BV$
entropy solutions of \eqref{1}, and then the general existence result in
$L^1\cap L^\infty$ follows by a density argument using the $L^1$-contraction.
\end{remark}


\section{Error estimates}\label{sec:framework}

In this section we give different error estimates and convergence
results for our schemes, 
estimates that are valid for general Levy measures and better
estimates that holds for fractional measures satisfying \eqref{fractional_meas}.
To give the general result, we need the following quantities:
\begin{align*}
I_1^{\epsilon,r}&=\frac{1}{\epsilon}\int_{|z|\leq r}|z|^2\,\dif\mu(z),\\
I_2^{\epsilon,\delta,r}&=\left(\frac{\Delta
x}{\epsilon}+\frac{\Delta
t}{\delta}\right)\Bigg(\int_{r<|z|\leq1}|z|\,\dif\mu(z)+\int_{|z|>1}\dif\mu(z)\Bigg);\\
I_3^{r}&=\mathcal{E}_{\Delta t}(\bar u)\int_{|z|>r}\dif\mu(z).
\end{align*}
\begin{theorem}\label{th:kuz_prior}\emph{(Error estimates)} Assume
  (A.2) -- (A.4) hold, and let $u$ be the entropy solution of \eqref{1}.
\smallskip

\noindent (a) Let $U$ be a solution of
the implicit method \eqref{scheme_implicit} and $\bar u$ defined by
\eqref{interp_impl}. Then for all $\epsilon>0$, $0<\delta<T$, and
$\frac{\Delta x}{2}<r\leq 1$, 
\begin{align}
\label{implicit-kuz}
\|u(\cdot,T)-\bar u(\cdot,T)\|_{L^{1}(\mathbb{R}^d)}\leq
C_T\bigg(\epsilon+\mathcal{E}_\delta(u)\vee
\mathcal{E}_\delta(\bar u)+I^{\epsilon,r}_1+I_2^{\epsilon,\delta,r}\bigg),
\end{align}

\noindent (b) Assume also \eqref{cfl} holds, and let $U$ be a solution of
the explicit method \eqref{scheme_explicit} and $\bar u$ defined by
\eqref{interp_expl}. Then for all
$\epsilon>0$, $0<\delta<T$, and $\frac{\Delta x}{2}<r\leq 1$,
\begin{equation}\label{explicit-kuz}
\begin{split}
\|u(\cdot,T)-\bar u(\cdot,T)\|_{L^{1}(\mathbb{R}^d)} \leq
C_T\bigg(\epsilon+\mathcal{E}_\delta(u)\vee\mathcal{E}_\delta(\bar u)+I^{\epsilon,r}_1+I_2^{\epsilon,\delta,r}+I_3^{r}\bigg),
\end{split}
\end{equation}
\end{theorem}

The proof of this result will be given in Section \ref{sec:pf}.

\begin{corollary}[Convergence]
Under the assumptions of Theorem \ref{th:kuz_prior}, the solutions of
the implicit method \eqref{scheme_implicit} and the explicit
method \eqref{scheme_explicit} both converge to the unique entropy solution
of \eqref{1} as $\dx,\dt\ra0$.
\end{corollary}
\begin{proof}
The result follows from the error estimates of Theorem
\ref{th:kuz_prior} by first sending 
$\dx,\dt\ra0$, then $r\ra0$, and finally $\eps,\delta\ra0$.
\end{proof}

We will now see how Theorem \ref{th:kuz_prior} (along with Lemma
\ref{lem:time-reg_frac}) can be used to produce explicit
rates of convergence for our scheme in the case of fractional measures
satisfying \eqref{fractional_meas}. First we define
\begin{equation}\label{rate1}
\sigma_{\lambda}^{IM}(\tau)=\left\{
\begin{array}{ll}
\tau^\frac{1}{2}&\lambda\in(0,1),\\
\tau^\frac{1}{2}|\log\tau|&\lambda=1,\\
\tau^\frac{2-\lambda}{2}&\lambda\in(1,2),
\end{array}
\right.
\end{equation}
and
\begin{equation}\label{rate2}
\sigma_{\lambda}^{EX}(\tau)=\left\{
\begin{array}{ll}
\tau^\frac{1}{2}&\lambda\in\left(0,\frac{2}{3}\right],\\
\tau^\frac{2-\lambda}{2+\lambda}&\lambda\in(\frac{2}{3},1)\cup(1,2).\\
\end{array}
\right.
\end{equation}

\begin{theorem}\emph{(Convergence rate for fractional measures)}\label{th:conv_frac_meas}
Under the assumptions of Lemma \ref{lem:time-reg_frac} (including
\eqref{fractional_meas} and a CFL condition for the implicit scheme),
for all $\lambda\in(0,2)$,
\begin{equation*}
\begin{split}
\|u(\cdot,T)-\bar u(\cdot,T)\|_{L^{1}(\mathbb{R}^d)}
\leq\left\{
\begin{split}
C_T\,\sigma_\lambda^{IM}(\Delta x)\qquad&\text{ for the implicit method
\eqref{scheme_implicit},}\\
&\\
C_T\,\sigma_\lambda^{EX}(\Delta x)\qquad&\text{ for the explicit method
\eqref{scheme_explicit}.}
\end{split}
\right.
\end{split}
\end{equation*}
\end{theorem}

Note that the rate for the explicit method is worse due to the extra
term $I_3^r$ in Theorem \ref{th:kuz_prior}.

\begin{corollary}[Explicit scheme when $\lambda=1$]
Let the assumptions of Lemma \ref{lem:time-reg_frac} (b) hold with
$\lambda=1$ and let $\alp\in(1,2)$ be arbitrary. If the stronger CLF condition
$C\frac{\dt}{\dx^\alp}<1$ holds, then
\begin{equation*}
\|u(\cdot,T)-\bar u(\cdot,T)\|_{L^{1}(\mathbb{R}^d)}
\leq
C_T\,\sigma_\alp^{EX}(\Delta x)\quad\text{ for the explicit method \eqref{scheme_explicit}.}
\end{equation*}
\end{corollary}
\begin{proof}
Note that the CFL condition \eqref{CFLj} is satisfied and that the
assumption \eqref{fractional_meas} holds with any
$\lambda\in[1,2)$. Hence the result follows from the $\lambda>1$ case
in Theorem \ref{th:conv_frac_meas}.
\end{proof}


\begin{proof}[Proof of Theorem \ref{th:conv_frac_meas}]
Let us first give the proof for the implicit method
\eqref{scheme_implicit}. First we note that by
\eqref{fractional_meas}, 
\begin{equation*}
\begin{split}
\int_{|z|\leq r}|z|^2\ \dif\mu(z)\leq c_\lambda\int_{|z|\leq
r}\frac{|z|^2}{|z|^{d+\lambda}}\ \dif z\leq
O\left(r^{2-\lambda}\right)\quad\text{for all $\lambda\in(0,2)$, $r\leq1$,}
\end{split}
\end{equation*}
while
\begin{equation*}
\begin{split}
\int_{r<|z|\leq 1}|z|\ \dif\mu(z)\leq c_\lambda\int_{r<|z|\leq
1}\frac{|z|}{|z|^{d+\lambda}}\ \dif z=
\begin{cases}
O(1)&\text{if }\lambda\in(0,1),\\
O(|\ln r|)&\text{if }\lambda=1,\\
O\left(r^{1-\lambda}\right)&\text{if }\lambda\in(1,2).
\end{cases}
\end{split}
\end{equation*}
Using these estimates along with the CFL condition \eqref{CFLjj} and
Lemma \ref{lem:time-reg_frac}, we
find that the estimate \eqref{implicit-kuz} in Theorem
\ref{th:kuz_prior} takes the form  
\begin{equation*}
\begin{split}
&\|u(\cdot,T)-\bar u(\cdot,T)\|_{L^{1}(\mathbb{R}^d)}\leq C_T\begin{cases}
\epsilon+\delta+\frac{r^{2-\lambda}}{\epsilon}+\left(\frac{\Delta
x}{\epsilon}+\frac{\Delta x}{\delta}\right)&\text{if
$\lambda\in(0,1)$,}\\[0.2cm]
\epsilon+\delta\,|\ln\delta|+\frac{r}{\epsilon}+|\ln
r|\left(\frac{\Delta x}{\epsilon}+\frac{\Delta
x}{\delta}\right)&\text{if
$\lambda=1$,}\\[0.2cm]
\epsilon+\delta^{\frac{1}{\lambda}}+\frac{r^{2-\lambda}}{\epsilon}+r^{1-\lambda}\left(\frac{\Delta
x}{\epsilon}+\frac{\Delta x^\lambda}{\delta}\right)&\text{if
$\lambda\in(1,2)$.}
\end{cases}
\end{split}
\end{equation*}
The conclusion then follows by taking $r=\Delta x$ for all
$\lambda\in(0,2)$, $\epsilon=\delta=\sqrt{\Delta x}$ for
$\lambda\in(0,1]$, while $\epsilon=\Delta x^{\frac{2-\lambda}{2}}$
and $\delta=\Delta x^{\frac{\lambda}{2}}$ for $\lambda\in(1,2)$.

For the explicit method \eqref{scheme_explicit} we also need to take
into account the extra $I_3$-term,
$$I^r_3=\sigma_\lambda(\Delta
t)\underbrace{\int_{|z|>r}\dif\mu(z)}_{O(r^{-\lambda})},$$
Lemma \ref{lem:time-reg_frac}, and the slightly more restrictive CFL
condition \eqref{CFLj}. The expression \eqref{explicit-kuz} in Theorem
\ref{th:kuz_prior} then takes the form 
\begin{equation*}
\begin{split}
&\|u(\cdot,T)-\bar u(\cdot,T)\|_{L^{1}(\mathbb{R}^d)}\\
&\leq C_T\begin{cases}
\epsilon+\delta+\frac{r^{2-\lambda}}{\epsilon}+\left(\frac{\Delta
x}{\epsilon}+\frac{\Delta x}{\delta}\right)+\frac{\Delta
x}{r^{\lambda}}&\text{if
$\lambda\in(0,1)$,}\\[0.2cm]
\epsilon+\delta^{\frac{1}{\lambda}}+\frac{r^{2-\lambda}}{\epsilon}+r^{1-\lambda}\left(\frac{\Delta
x}{\epsilon}+\frac{\Delta x^\lambda}{\delta}\right)+\frac{\Delta
x}{r^{\lambda}}&\text{if $\lambda\in(1,2)$.}
\end{cases}
\end{split}
\end{equation*}
We minimize two and two terms and take the maximum minimizers, first
w.r.t. $\eps$ and $\delta$ and then w.r.t. $r$, 
\begin{equation*}
\begin{split}
&\|u(\cdot,T)-\bar u(\cdot,T)\|_{L^{1}(\mathbb{R}^d)}\\
&\leq C_T\begin{cases}
r^{\frac{2-\lambda}2}+\Delta x^{\frac12}+\frac{\Delta
x}{r^{\lambda}},&\text{if }\lambda\in(0,1)\\[0.2cm]
r^{\frac{2-\lambda}2}+r^{\frac{1-\lambda}2}\Delta
x^{\frac12}+r^{\frac{1-\lambda}{1+\lambda}}\Delta x^{\frac\lambda{1+\lambda}}+\frac{\Delta
x}{r^{\lambda}}&\text{if } \lambda\in(1,2),
\end{cases}\\
&\leq  C_T\begin{cases}
\Delta x^{\frac12}+\Delta x^{\frac{2-\lambda}{2+\lambda}}&\text{if }\lambda\in(0,1),\\[0.2cm]
\dx^{\frac{2-\lambda}2}+\Delta x^{\frac{2-\lambda}{3-\lambda}}+\Delta x^{\frac{2-\lambda}{2+\lambda}}&\text{if } \lambda\in(1,2).
\end{cases}
\end{split}
\end{equation*}
The final result follows since
$\frac{2-\lambda}{3-\lambda}>\frac{2-\lambda}{2+\lambda}$ for
$\lambda\in(\frac12,2)$ and $\frac{2-\lambda}{2+\lambda}<\frac12$ for $\lambda\in(\frac23,2)$.
\end{proof}

 \begin{remark}\label{rem:disc_cut}
The rates can not be improved by taking a different truncation of the
singularity, i.e. replacing in the method
\begin{equation*}
\begin{split}
G_{\alpha,\beta}&=\frac1{\dx^d}\int_{R_{\alpha}}\int_{|z|>\frac{\Delta
x}{2}}\mathbf{1}_{R_\beta}(x+z)-\mathbf{1}_{R_\beta}(x)\
\dif\mu(z)\,\dif x
\end{split}
\end{equation*}
by
\begin{equation*}
\begin{split}
G_{\alpha,\beta}&=\frac1{\dx^d}\int_{R_{\alpha}}\int_{|z|>\rho_\lambda(\Delta
x)}\mathbf{1}_{R_\beta}(x+z)-\mathbf{1}_{R_\beta}(x)\
\dif\mu(z)\,\dif x.
\end{split}
\end{equation*}
The reason is that the function $\rho_\lambda$ that minimize
the error expression
\begin{equation*}
\begin{split}
\epsilon+\delta+\frac{\rho^{2-\lambda}_\lambda(\Delta
x)}{\epsilon}+\rho^{1-\lambda}_\lambda(\Delta x)\left(\frac{\Delta
x}{\epsilon}+\frac{\Delta x}{\delta}\right),
\end{split}
\end{equation*}
is always $\rho_\lambda(\Delta x)=O(\Delta x)$!
\end{remark}

\begin{remark}
We believe that the rates for the implicit schemes are optimal, at
least when there are nonlinear convection terms in the equation
(i.e. when $f\neq0$ in \eqref{1}, see Section
\ref{sec:convection_equations}). But we 
 have not found analytical examples confirming this, nor have we been
 able to observe the above rates in preliminary, but probably too crude,
 numerical tests. 
Maybe it is not straight forward to construct
analytical or numerical examples confirming the optimality of the
rates. We leave 
it as a challenge for people with more experience in realizing
numerical schemes to test the optimality numerically.
\end{remark}

\section{Convection-diffusion equations}\label{sec:convection_equations}
In this section we discuss how to extend the results established in
the previous sections to the case $f\neq 0$. Note that all the arguments needed
to handle the 
additional $f$-term are well-known. We consider the following numerical methods
\begin{align}
U_{\alpha}^{n+1}&=U_{\alpha}^{n}+\Delta t\sum_{l=1}^dD_l^{-}\hat
f_l(U_{\alpha}^{n+1},U_{\alpha+e_l}^{n+1})
+\Delta t\, \hat\Levy^\mu\langle A(U^{n+1})\rangle_\alpha,&&\text{(implicit)}\label{scheme_implicit_0}\\
U_{\alpha}^{n+1}&=U_{\alpha}^{n}+\Delta t\sum_{l=1}^dD_l^{-}\hat
f_l(U_{\alpha}^{n},U_{\alpha+e_l}^{n})
+\Delta t\, \hat\Levy^\mu\langle A(U^{n+1})\rangle_\alpha,&&\text{(expl-impl)}\label{scheme_implicit_1}\\
U_{\alpha}^{n+1}&=U_{\alpha}^{n}+\Delta t\sum_{l=1}^dD_l^{-}\hat
f_l(U_{\alpha}^{n},U_{\alpha+e_l}^{n}) +\Delta t\,
\hat\Levy^\mu\langle
A(U^{n})\rangle_\alpha,&&\text{(explicit)}\label{scheme_explicit_2}
\end{align}
where
\begin{itemize}
\item[\emph{(i)}] $D^{-}_lU_\alpha=\frac{1}{\Delta
    x}(U_\alpha-U_{\alpha-e_l})$ and
$\{e_l\}_l$ is the standard basis of $\R^d$, and
\smallskip
\item[\emph{(ii)}] $\hat f=(\hat f_1,\dots,\hat f_d)$ is a consistent
  (i.e. $\hat {f}(u,u)=f(u)$), 
  Lipschitz continuous numerical flux which is non-decreasing 
w.r.t.~the first variable and non-increasing w.r.t.~the second one.
\end{itemize}

\begin{remark}
Some examples of numerical fluxes $\hat f$ satisfying $(ii)$ are the
well-known Lax-Friedrichs flux, the 
Godunov flux, and the Engquist-Osher flux, cf.~e.g.~\cite{Kr:Book}.
\end{remark}

For the schemes \eqref{scheme_implicit_1} and
\eqref{scheme_explicit_2}, we also need the CFL conditions
\begin{align}\label{iki}
&2d\,L_F\frac{\Delta t}{\Delta x}+\bar c L_A\,\frac{\Delta
  t}{\hat\sigma_\mu(\Delta x)}<1\qquad\text{and}\qquad 2dL_F\frac{\Delta t}{\Delta x}<1
\end{align}
respectively (compare with \eqref{cfl}), where $\hat\sigma_\mu$
is defined in \eqref{sigma0} and $L_F$ is the Lipschitz constant of
$\hat f$. Then the all the a priori estimates and other results of Section 
\ref{sec:comp} continue to hold for the new schemes, and we still have
compactness via Kolmogorov's theorem. The modifications needed 
to identify the any limit as the unique entropy solution of \eqref{1} are
standard and can be found e.g. in Chapter 3 in \cite{Holden/Risebro},
and hence the convergence of the methods
\eqref{scheme_implicit_0}--\eqref{scheme_explicit_2} follows.

We will now give the statement of the result of Theorem
\ref{th:kuz_prior} that is valid for the current setting where $f\neq
0$. To do so we reuse the quantities $I_1^{\epsilon,r}$ and $I_3^{r}$
of section \ref{sec:framework}, but redefine $I_2^{\epsilon,\delta,r}$ as follows
\begin{equation*}
I_2^{\epsilon,\delta,r}=\left(\frac{\Delta
x}{\epsilon}+\frac{\Delta t}{\delta}\right)\bigg(1+\int_{ r<|z|\leq
1}|z|\,\dif\mu(z)+\int_{|z|>1}\dif\mu(z)\bigg).
\end{equation*}

\begin{theorem}\label{th:extended}\emph{(Error estimates)} Assume
  (A.1) -- (A.4) hold, and let $u$ be the entropy  solution of \eqref{1}. 
\smallskip

\noindent (a) Let $U$ be a solution of \eqref{scheme_implicit_0} {\em or}
\eqref{scheme_implicit_1} and $\bar u$ defined by 
\eqref{interp_impl}. For \eqref{scheme_implicit_1} we also
need the second CLF condition in \eqref{iki}. Then for all
$\epsilon>0$, $0<\delta<T$, and 
$\frac{\Delta x}{2}<r\leq 1$, 
\begin{align*}
\|u(\cdot,T)-\bar u(\cdot,T)\|_{L^{1}(\mathbb{R}^d)}\leq
C_T\bigg(\epsilon+\mathcal{E}_\delta(u)\vee
\mathcal{E}_\delta(\bar u)+I^{\epsilon,r}_1+I_2^{\epsilon,\delta,r}\bigg).
\end{align*}

\noindent (b) Assume also that the first CFL condition in \eqref{iki}
holds, and let $U$ be a solution of \eqref{scheme_explicit_2} and $\bar
u$ defined by \eqref{interp_expl}. Then for all
$\epsilon>0$, $0<\delta<T$, and $\frac{\Delta x}{2}<r\leq 1$,
\begin{equation*}
\|u(\cdot,T)-\bar u(\cdot,T)\|_{L^{1}(\mathbb{R}^d)} \leq
C_T\bigg(\epsilon+\mathcal{E}_\delta(u)\vee\mathcal{E}_\delta(\bar u)+I^{\epsilon,r}_1+I_2^{\epsilon,\delta,r}+I_3^{r}\bigg).
\end{equation*}
\end{theorem}

The proof is essentially equal to the proof of Theorem
\ref{th:kuz_prior} augmented by standard Kuznetsov type computations to handle
the $f$-term, cf. e.g. \cite[Example 3.14]{Holden/Risebro}. We skip it.

\begin{remark}
It is easy to see that the contribution to the error from the
  discretization of the $f$-term  
is always less or of the same order as the contributions of the
other terms. In particular, for fractional measures
\eqref{fractional_meas}, we immediately get that the schemes
satisfy the error estimate of Theorem
\ref{th:conv_frac_meas}  with modulus $\sigma_{\lambda}^{IM}$ for
\eqref{scheme_implicit_0}  and \eqref{scheme_implicit_1} and modulus
$\sigma_{\lambda}^{EX}$ for \eqref{scheme_explicit_2}.
\end{remark}

\section{The proof of Theorem \ref{th:kuz_prior}}
\label{sec:pf}

\begin{proof}[Proof of Theorem \ref{th:kuz_prior} for the implicit
  method \eqref{scheme_implicit}] \
\smallskip

\noindent $1.\quad$ We use Lemma \ref{lem:kuznetsov} to
  compare the solution of the scheme to the exact solution. In the
  resulting inequality, we introduce the scheme via the time
  derivative and the initial/final terms. To do this, we use
  integration by parts on each interval 
  $(t_n,t_{n+1})$ and summation by parts to get discrete time
  derivatives on $\bar u$ so that we can use the cell entropy inequality
  \eqref{j4}. We get that (remember the definition of
  $\bar u$) 
\begin{align*}
&-\iint_{Q_T}\iint_{Q_T}\eta(\bar
u(x,t),u(y,s))\, \partial_t\varphi^{\epsilon,\delta}(x,y,t,s)\dif w+\text{initial and final terms}\\
&=\iint_{Q_T}\sum_{n=0}^{N-1}\sum_{\alp\in\mathbb{Z}^d}\bigg(\eta(U_{\alp}^{n+1},u(y,s))-\eta(U_{\alp}^{n},u(y,s))\bigg)\int_{R_{\alpha}}
\varphi^{\epsilon,\delta}(x,y,t_{n+1},s)\, \dif x\, \dif y\dif s.
\end{align*}
Let $\bar\varphi^{\epsilon,\delta}=\bar\varphi^{\epsilon,\delta}(x,y,t,s)$
be the function which for each $(y,s)\in Q_T$  is defined by
\begin{equation*}
\begin{split}
\varphi_\alpha^n=\frac{1}{\Delta
x^d}\int_{R_{\alpha}}\varphi^{\epsilon,\delta}(x,y,t_{n},s)\ \dif
x\quad\text{for}\quad x\in R_\alp,\ t\in(t_{n-1},t_n],
\end{split}
\end{equation*}
and use above equation along with the cell entropy inequality
\eqref{j4} and Lemma \ref{lemLG} to write the inequality of Lemma
\ref{lem:kuznetsov} in the following way
\begin{align*}
&\|u(\cdot,T)-\bar u(\cdot,T)\|_{L^{1}(\mathbb{R}^d)}\leq C_T\,(\Delta
x+\epsilon+\mathcal{E}_\delta(u)\vee\mathcal{E}_\delta(v))\\  
&\quad+\underbrace{\iint_{Q_T}\iint_{Q_T}\eta(A(\bar
  u(x,t)),A(u(y,s)))\,
  \Levy^{\mu^\ast}_r[\varphi^{\epsilon,\delta}(x,\cdot,t,s)](y)\ \dif
  w}_{H_1}\\ 
&\quad+\underbrace{\iint_{Q_T}\iint_{Q_T}\eta(A(\bar
  u(x,t)),A(u(y,s)))\,
  \hat\Levy^{\mu^\ast}_r[\bar\varphi^{\epsilon,\delta}(\cdot,y,t,s)](x)\
  \dif w}_{H_2}\\ 
&\quad+\underbrace{\iint_{Q_T}\iint_{Q_T}\eta'(\bar u(x,t),u(y,s))\,
  \Levy^{\mu,r}[A(\bar u(\cdot,t))](x)\,
  (\bar\varphi^{\epsilon,\delta}-\varphi^{\epsilon,\delta})(x,y,t,s)\
  \dif w}_{H_3}\\ 
&\quad+\underbrace{\iint_{Q_T}\iint_{Q_T}\eta(A(\bar
u(x,t)),A(u(y,s)))\,\gamma^{\mu^\ast,r}\cdot(\hat
D\bar\varphi^{\epsilon,\delta}-\nabla_x\varphi^{\epsilon,\delta})(x,y,t,s)\
\dif w}_{H_4}. 
\end{align*}
Here we have also used the notation 
$$\hat\Levy[\phi](x)=\hat\Levy_r[\phi](x)+\hat\Levy^r[\phi](x)+\gamma^{\mu,r}\cdot \hat
D_{\Delta x}\phi(x)$$ 
where $\hat\Levy$ is defined in \eqref{newB}, $\hat\Levy^r=\Levy^r$
for $r\geq\frac\dx2$, and
$$\hat\Levy_r[\phi](x)=\int_{\frac\dx
  2<|z|<r}\phi(x+z)-\phi(x)-\mathbf{1}_{|z|<1}z\cdot \hat
D_{\Delta x}\phi(x)\ \dif\mu(z).$$
Note that the discrete operator $\hat D_l=\hat D_{\dx,l}$ (see
\eqref{diff_operator}) always acts on the $x$-variable (the variable
of $\bar u$). To complete the 
proof we need to estimate $H_1,\dots,H_4$.
\smallskip

\noindent $2.\quad$ \emph{Estimates of $H_1$ and  $H_2$.} 
By Taylor's formula with integral remainder,
integration by parts, and Fubini (-- see
e.g. Lemma B.1 in \cite{Alibaud/Cifani/Jakobsen} for more details), 
\begin{equation*}
\begin{split}
|H_1|&\leq\iint_{Q_{T}}\iint_{Q_{T}}\int_{|z|\leq r}\int_0^1(1-\tau)\,\big|D_y\eta(A(\bar u(x,t)),A(u(y,s)))\big|\\
&\qquad\qquad\qquad\qquad
\cdot\omega_\delta(t-s)\,\underbrace{\big|D_{y}\Omega_{\epsilon}(x-y+\tau
z)\big|}_{=|D_{x}\Omega_{\epsilon}(x-y+\tau z)|}
|z|^2\ \dif \tau\, \dif\mu(z)\,\dif w\\
&\leq\frac12 L_A\int_{0}^T|u(\cdot,s)|_{BV(\R^d)}\,\dif s\int_{\R^d}|D_{x}\Omega_{\epsilon}(x)|\ \dif x\int_\R\omega_\eps(t)\,\dif t\int_{|z|\leq r}|z|^2\ \dif\mu(z))\\
&\leq
C_T\,L_A\,|u_0|_{BV(\R^d)}\,\epsilon^{-1}\int_{|z|\leq
r}|z|^2\,\dif\mu(z).
\end{split}
\end{equation*}
Here we also used Theorem \ref{thm:WP} and the standard estimate
$\int_{\R^d}|D_{x}\Omega_{\epsilon}(x)|\ \dif x=\mathcal O(\frac1\eps)$.

We find a similar estimate for $H_2$ via a regularization procedure and
the argument for $H_1$ above. Let $\bar\varphi^{\epsilon,\delta}_\varrho$ be a
mollification in the $x$-variable of $\bar\varphi^{\epsilon,\delta}$, i.e.
$\bar\varphi^{\epsilon,\delta}_\varrho=\bar\varphi^{\epsilon,\delta}\ast_x\Omega_\rho$
where the convolution is in $x$ only. Then
$\bar\varphi^{\epsilon,\delta}_\varrho$ is smooth in $x$, and
\begin{equation*}
\begin{split}
|\bar\varphi^{\epsilon,\delta}_\varrho(\cdot,y,t,s)|_{BV(\R^d)}\leq|\bar\varphi^{\epsilon,\delta}(\cdot,y,t,s)|_{BV(\R^d)}
\leq|\varphi^{\epsilon,\delta}(\cdot,y,t,s)|_{BV(\R^d)}=O\left(\epsilon^{-1}\right),
\end{split}
\end{equation*}
where the first inequality holds for all $\varrho$ small enough
(cf. e.g. \cite[Theorem 5.3.1]{Ziemer}), while the second one is
obvious. Let us call 
\begin{equation*}
\begin{split}
H_2^{\varrho}=\iint_{Q_T}\iint_{Q_T}\eta(A(\bar u(x,t)),A(u(y,s)))\,
\hat\Levy^{\mu^\ast}_r[\bar\varphi^{\epsilon,\delta}_\varrho(\cdot,y,t,s)](x)\
\dif w.
\end{split}
\end{equation*}
First note that
$\lim_{\varrho\rightarrow0}H_2^{\varrho}=H_2$ by the dominated
convergence theorem since we are integrating away from the singularity and 
$\bar\varphi^{\epsilon,\delta}_\varrho(\cdot,y,t,s)\ra\bar\varphi^{\epsilon,\delta}(\cdot,y,t,s)$
pointwise. Then, since 
$\bar\varphi^{\epsilon,\delta}_\varrho(\cdot,y,t,s)$ is
smooth, we repeat the argument used for $H_1$ and obtain
\begin{align*}
|H_2^\varrho|&\leq C_T\,L_A\,|u_0|_{BV(\R^d)}\,|\bar\varphi^{\epsilon,\delta}_\varrho|_{BV(\R^d)}\int_{\frac{\Delta x}{2}<|z|\leq r}|z|^2\,\dif\mu(z).
\end{align*}
Since
$|\bar\varphi^{\epsilon,\delta}_\varrho|_{BV(\R^d)}=O(\eps^{-1})$, we
can take the limit $\varrho\rightarrow0$ and get
\begin{align*}
|H_2|&\leq
C_T\,L_A\,|u_0|_{BV(\R^d)}\,\epsilon^{-1}\int_{|z|\leq
r}|z|^2\,\dif\mu(z).
\end{align*}
\smallskip

\noindent $3.\quad$\emph{Estimate of $H_3$.} By the definition of
$\bar\varphi^{\epsilon,\delta}$ and properties of 
mollifiers, a standard argument shows that
\begin{align*}
&\iint_{Q_T}\big|\bar\varphi^{\epsilon,\delta}(x,y,t,s)-\varphi^{\epsilon,\delta}(x,y,t,s)\big|\ \dif
y\,\dif s\\
&\leq d|\Omega_\eps|_{BV}\|\omega_\delta\|_{L^1}\dx+d\|\Omega_\eps\|_{L^1}|\omega_\delta|_{BV}\dt\leq O\left(\frac{\Delta x}{\epsilon}+\frac{\Delta
t}{\delta}\right).
\end{align*}
Similar estimates are given in
e.g. \cite{Cifani/Jakobsen/Karlsen}. This estimate along with several 
applications of Fubini's theorem then show that for all $\frac{\Delta
  x}{2}<r\leq 1$, 
\begin{align*}
|H_3|&\leq\iint_{Q_T}\big|\Levy^{\mu,r}[A(\bar u(\cdot,t))](x)\big|\,\Bigg(\iint_{Q_T}\big|\bar\varphi^{\epsilon,\delta}(x,y,t,s)
-\varphi^{\epsilon,\delta}(x,y,t,s)\big|\ \dif y\,\dif s\Bigg)\,\dif x\,\dif t\\
&\leq c\,L_A\left(\frac{\Delta x}{\epsilon}+\frac{\Delta t}{\delta}\right)\Bigg(\iint_{Q_T}\int_{r<|z|\leq1}|\bar u(x+z,t)-\bar u(x,t)|\ \dif\mu(z)\,\dif x\,\dif t\\
&\qquad\qquad\qquad\qquad\qquad\qquad+\iint_{Q_T}\int_{|z|>1}|\bar
u(x+z,t)-\bar u(x,t)|\ \dif\mu(z)\,\dif x\,\dif t\Bigg)\\
&\leq C_TL_A\Big(\frac{\Delta
x}{\epsilon}+\frac{\Delta
x}{\delta}\Big)\left(|u_0|_{BV}\int_{r<|z|\leq
1}|z|\,\dif\mu(z)+\|u_0\|_{L^1}\int_{|z|>1}\dif\mu(z)\right).
\end{align*}

\noindent 4.$\quad$\emph{Estimate of $H_4$.} Let $l\in(0,\ldots,d)$ and write
\begin{align*}
H_{4,l}
&=\gamma^{\mu^{\ast},r}_l\underbrace{\iint_{Q_T}\sum_{\alpha\in\Z^d}\sum_{n=0}^{N-1}\eta(A(U^n_{\alpha}),A(u(y,s)))\int_{t_n}^{t_{n+1}}\int_{R_\alpha}\hat
D_l\bar\varphi^{\epsilon,\delta}(x,y,t,s)\
\dif w}_{H_{4,l}^1}\\
&\quad-\gamma^{\mu^{\ast},r}_l\underbrace{\iint_{Q_T}\sum_{\alpha\in\Z^d}\sum_{n=0}^{N-1}\eta(A(U^n_{\alpha}),A(u(y,s)))\int_{t_n}^{t_{n+1}}\int_{R_\alpha}
\partial_{x_l}\varphi^{\epsilon,\delta}(x,y,t,s)\
\dif w}_{H_{4,l}^2}.
\end{align*}
Since $\int_{t_n}^{t_{n+1}}\int_{R_\alpha}
\bar\varphi^{\epsilon,\delta}(x,y,t,s)\,\dif x\dif t=\int_{t_n}^{t_{n+1}}\int_{R_\alpha}
\varphi^{\epsilon,\delta}(x,y,t_{n+1},s)\,\dif x\dif t$ by definition,
we can use summation by parts to find that
\begin{align*}
H_{4,l}^1&=-\iint_{Q_T}\sum_{\alpha\in\Z^d}\sum_{n=0}^{N-1}\hat
D_l\eta(A(U^n_{\alpha}),A(u(y,s)))\int_{t_n}^{t_{n+1}}\int_{R_\alpha}
\varphi^{\epsilon,\delta}(x,y,t_{n+1},s)\ \dif w.
\end{align*}
Integration in the $x_l$-direction followed by summation by parts
leads to
\begin{align*}
H_{4,l}^2&=-\Delta
x\iint_{Q_T}\sum_{\alpha\in\Z^d}\sum_{n=0}^{N-1}\hat
D_l\eta(A(U^n_{\alpha}),A(u(y,s)))\\
&\qquad\cdot\int_{t_n}^{t_{n+1}}\idotsint
\varphi^{\epsilon,\delta}(\xl,y,t,s)\ \dif x_1\ldots \dif
x_{l-1}\,\dif x_{l+1}\ldots\dif x_d\,\dif t\,\dif
y\,\dif s.
\end{align*}
Here we first integrated 
$\partial_{x_l}\varphi^{\epsilon,\delta}(\cdot,y,t,s)$ along the
interval $(x_{\alpha_l},x_{\alpha_{l+1}})$ to obtain the difference
$\varphi^{\epsilon,\delta}(x_{|x_l=x_{\alp_{l+1}}},y,t,s)-\varphi^{\epsilon,\delta}(\xl,y,t,s)$, 
and then we used summation by parts to move this difference onto
$\eta(A(U^n_{\alpha}),A(u(y,s)))$. Note that 
$\xl=(x_1,\dots,x_{l-1},x_{\alp_l},x_{l+1},\dots, x_d)$, and that
$x_l=x_{\alpha_l}$ is fixed here while the other variables $x_j,\
j\neq l$ vary. 

By the above computations,  the inequality $|\hat
D_l\eta(A(U^n_{\alpha}),A(u(y,s)))|\leq |\hat D_lA(U^n_{\alpha})|$
(i.e.~$||a-k|-|b-k||\leq|a-b|$), and Fubini, we find that
\begin{equation*}
\begin{split}
H_{4,l}=&\ \gamma^{\mu^{\ast},r}_l\iint_{Q_T}\sum_{\alpha\in\Z^d}\sum_{n=0}^{N-1}\hat
D_l\eta(A(U^n_{\alpha}),A(u(y,s)))\\
&\cdot\int_{t_n}^{t_{n+1}}\int_{R_\alp}\left(
\varphi^{\epsilon,\delta}(\xl,y,t,s)-\varphi^{\epsilon,\delta}(x_\alpha,y,t_{n+1},s)
\right)\dif x\,\dif
t\ \dif y\,\dif s\\
\leq&\ \gamma^{\mu^{\ast},r}_l\sum_{\alpha\in\Z^d}\sum_{n=0}^{N-1}\,|\hat D_lA(U^n_{\alpha})|\\
&\cdot\int_{t_n}^{t_{n+1}}\int_{R_\alp}\iint_{Q_T}\left|
\varphi^{\epsilon,\delta}(\xl,y,t,s)-\varphi^{\epsilon,\delta}(x_\alpha,y,t_{n+1},s)
\right|\dif w.
\end{split}
\end{equation*}
Since
$\phi^{\eps,\delta}(x,y,t,s)=\Omega_\eps(x-y)\omega_\delta(t-s)$ and
$(x,t)\in R_\alp\times (t_n,t_{n+1}]$, we
find as in part 3 that
\begin{align*}
&\iint_{Q_T}\left|
\varphi^{\epsilon,\delta}(\xl,y,t,s)-\varphi^{\epsilon,\delta}(x_\alpha,y,t_{n+1},s)
\right|\,\dif y\dif s\\
&\leq
C\Big(|\Omega_\eps|_{BV}\|\omega_\delta\|_{L^1}\dx+\|\Omega_\eps\|_{L^1}|\omega_\delta|_{BV}\dt\Big)=O\left(\frac\dx\eps+\frac\dt\delta\right).
\end{align*}
Summing over $l$ we then find that
\begin{align*}
|H_4|\leq d\,C|\gamma^{\mu^\ast,r}|\left(\frac{\Delta t}{\delta}+\frac{\Delta
x}{\epsilon}\right)\bigg(\sum_{n=0}^{N-1}\sum_{\alpha\in\Z^d}|\hat
D_lA(U^n_{\alpha})|\Delta t\,\Delta
x^{d}\bigg), 
\end{align*}
and since $\sum_{\alpha\in\Z^d}|\hat
D_lA(U^n_{\alpha})|\Delta
x^{d}=|A(\bar
u(\cdot,t_n))|_{BV}\leq L_A|u_0|_{BV}$, we conclude that
\begin{align*}
&|H_4|\leq C_TL_A\left(\frac{\Delta
x}{\epsilon}+\frac{\Delta t}{\delta}\right)\int_{r<|z|\leq
1}|z|\,\dif\mu(z).
\end{align*}
In view of part 1 - 4 the proof is now complete.
\end{proof}

\begin{proof}[Proof of Theorem \ref{th:kuz_prior} for the explicit method \eqref{scheme_explicit}]
We argue as in the beginning of the proof for the implicit method,
replacing the implicit cell entropy inequality by the explicit one
\eqref{u1}, and find that
\begin{align*}
&\|u(\cdot,T)-\bar u(\cdot,T)\|_{L^{1}(\mathbb{R}^d)}\leq C_T\,(\Delta x+\epsilon+\mathcal{E}_\delta(u)\vee\mathcal{E}_\delta(v))\\
&\quad+\iint_{Q_T}\iint_{Q_T}\eta(A(\bar u(x,t)),A(u(y,s)))\, \Levy^{\mu^\ast}_r[\varphi^{\epsilon,\delta}(x,\cdot,t,s)](y)\ \dif w\\
&\quad+\iint_{Q_T}\iint_{Q_T}\eta(A(\bar u(x,t)),A(u(y,s)))\,\hat\Levy^{\mu^\ast}_r[\bar\varphi^{\epsilon,\delta}(\cdot,y,t,s)](x)\ \dif w\\
&\quad+\iint_{Q_T}\iint_{Q_T}\eta'(\bar u(x,t+\Delta t),u(y,s))\, \Levy^{\mu,r}[A(\bar u(\cdot,t))](x)\,\bar\varphi^{\epsilon,\delta}(x,y,t,s)\ \dif w\\
&\quad-\iint_{Q_T}\iint_{Q_T}\eta'(\bar u(x,t),u(y,s))\,
\Levy^{\mu,r}[A(\bar u(\cdot,t))](x)\,
\varphi^{\epsilon,\delta}(x,y,t,s)\ \dif w\\
&\quad+\iint_{Q_T}\iint_{Q_T}\eta(A(\bar
u(x,t)),A(u(y,s)))\,\gamma^{\mu^\ast,r}\cdot(\hat
D\bar\varphi^{\epsilon,\delta}-\nabla_x\varphi^{\epsilon,\delta})(x,y,t,s)\
\dif w.
\end{align*}
The difference with the previous proof is the interpolation \eqref{interp_expl},
and more importantly, the new $\Levy^{\mu,r}$-terms. Note that by a
change of variables,
\begin{equation*}
\begin{split}
&\iint_{Q_T}\iint_{Q_T}\eta'(\bar u(x,t+\Delta t),u(y,s))\,
\Levy^{\mu,r}[A(\bar
u(\cdot,t))](x)\,\bar\varphi^{\epsilon,\delta}(x,y,t,s)\ \dif w\\
&= \iint_{Q_T}\iint_{Q_{\dt,T}}\eta'(\bar u(x,t),u(y,s))\,
\Levy^{\mu,r}[A(\bar u(\cdot,t-\Delta
t))](x)\,\bar\varphi^{\epsilon,\delta}(x,y,t,s)\ \dif
w\\
&\quad+\int_T^{T+\dt}\int_{\R^d}\iint_{Q_T}\eta'(\bar
u(x,t),u(y,s))
\Levy^{\mu,r}[A(\bar
u(\cdot,t-\dt))](x)\,\bar\varphi^{\epsilon,\delta}(x,y,t,s)\ \dif w,
\end{split}
\end{equation*}
where $Q_{a,b}=\R^d\times(a,b)$. The last term on the right
can be estimated by 
\begin{align*}
&\dt\|\phi^{\eps,\delta}\|_{L^1}\left(|A(\bar
  u)|_{BV}\int_{r<|z|<1}|z|\,\dif\mu(z)+2\|A(\bar
  u)\|_{L^1}\int_{|z|>1}\dif\mu(z)\right)\\
&=O\left(\dt\right)\int_{r<|z|<1}|z|\,\dif\mu(z).
\end{align*}
By similar computations, we can write the $\Levy^{\mu,r}$-terms in the
above inequality as
\begin{equation*}
\begin{split}
&\underbrace{\iint_{Q_T}\iint_{Q_{\dt,T}}\eta'(\bar u(x,t),u(y,s))\,
\Levy^{\mu,r}[A(\bar u(\cdot,t-\Delta
t))-A(\bar u(\cdot,t))](x)\,\bar\varphi^{\epsilon,\delta}(x,y,t,s)\ \dif
w}_{I}\\
&+\iint_{Q_T}\iint_{Q_T}\eta'(\bar u(x,t),u(y,s))\,
\Levy^{\mu,r}[A(\bar u(\cdot,t))](x)\,
(\bar\varphi^{\epsilon,\delta}-\varphi^{\epsilon,\delta})(x,y,t,s)\
\dif w\\
&+O\left(\dt\right)\int_{r<|z|<1}|z|\,\dif\mu(z).
\end{split}
\end{equation*}
Here we estimate the first term using the time regularity of $\bar u$,
\begin{equation*}
\begin{split}
I&\leq \iint_{Q_T}|\Levy^{\mu,r}[A(\bar u(\cdot,t-\Delta t))-A(\bar u(\cdot,t))](x)|\underbrace{\iint_{Q_T}\bar\varphi^{\epsilon,\delta}(x,y,t,s)\ \dif w}_{=O(1)}\\
&\leq c\, 2\,L_A\left(\int_{0}^T\|\bar u(\cdot,t-\Delta t)-\bar u(\cdot,t)\|_{L^1(\R^d)}\ \dif t\right)\int_{|z|>r}\dif\mu(z)\\
&\leq C_T\,\mathcal{E}_{\Delta t}(\bar u)\int_{|z|>r}\dif\mu(z),
\end{split}
\end{equation*}
where $\mathcal{E}_{\Delta t}(\bar u)$ is defined in
\eqref{time_mod}. Now all the remaining terms  can
be estimated as in the proof for the implicit method
\eqref{scheme_implicit}, so the proof is complete. 
\end{proof}




\end{document}